\documentclass[11pt,leqno]{amsart}
\makeatletter
\usepackage{amsfonts,delarray,amssymb,amsmath,amsthm,a4,a4wide}
\usepackage{color}

\usepackage[margin=0.94in]{geometry} 

\title[Boundary Harnack for the linearized Monge-Amp\`ere equation]{Boundary Harnack inequality for the linearized Monge-Amp\`ere equations and applications}
\author{Nam Q. Le}
\address{Department of Mathematics, Indiana University, 831 E 3rd St,
Bloomington, IN 47405, USA}
\email{nqle@indiana.edu}
\thanks{The research of the author was supported in part by the National Science Foundation under grant DMS-1500400.}
\newcommand{\review}[2][\right]{\relax
\ifx#1\right\relax \left.\fi#2#1\rvert}
\let\abs=\envert
 
\newtheorem{theorem}{Theorem}[section]

\newtheorem{propo}[theorem]{Proposition}
\newtheorem{remark}[theorem]{Remark}
\newtheorem{cor}[theorem]{Corollary}
\newtheorem{lemma}[theorem]{Lemma}
\newcommand{\bef}{\begin{flushright}}
\newcommand{\eef}{\end{flushright}}

\newcommand{\eval}[2][\right]{\relax
\ifx#1\right\relax \left.\fi#2#1\rvert}

\let\abs=\envert

\numberwithin{equation}{section}

\newcommand{\p}{\partial}
\newcommand\e{\varepsilon}

\def\beq{\begin{eqnarray*}}
\def\eeq{\end{eqnarray*}}

\def\RR{\mbox{$I\hspace{-.06in}R$}}

\newenvironment{myindentpar}[1]%
{\begin{list}{}%
         {\setlength{\leftmargin}{#1}}%
         \item[]%
}
{\end{list}}
\makeatletter
\def\l@subsection{\@tocline{2}{0pt}{2.5pc}{5pc}{}}
\makeatother
 
\begin{document}
\begin{abstract}
 In this paper, we obtain
 boundary Harnack estimates and comparison theorem for nonnegative solutions to the linearized Monge-Amp\`ere equations under natural assumptions 
on the domain, Monge-Amp\`ere measures and boundary data. 
Our results are boundary versions of Caffarelli and Guti\'errez's interior Harnack 
inequality for the
linearized Monge-Amp\`ere equations. 
As an application, we obtain sharp upper bound and global $L^p$-integrability for the Green's function of the linearized Monge-Amp\`ere operator.
\end{abstract}

\maketitle
\pagenumbering{arabic}

\section{Introduction }
\label{main_res}

This paper is concerned with boundary Harnack estimates for solutions to the linearized Monge-Amp\`ere equations 
$$L_\phi u:=-\sum_{i, j=1}^n \Phi^{ij} u_{ij}=0$$ 
where $\Phi =(\Phi^{ij})\equiv (\det D^2 \phi) (D^2 \phi)^{-1}$ denotes the cofactor matrix of the Hessian matrix $D^2\phi$ of a strictly convex function $\phi$. 
The operator $L_\phi$ appears in many contexts including the affine maximal surface equation in affine geometry \cite{TW00, TW05, TW08}, the Abreu's equation in the problem of finding
K\"ahler metrics of constant scalar curvature
in complex geometry \cite{Ab, D2, D4}, and the semigeostrophic equations in fluid mechanics \cite{Br, CNP, Loe}. 

The regularity theory for the linearized Monge-Amp\`ere equations was initiated in the fundamental paper \cite{CG} by Caffarelli and Guti\'errez. They developed an interior Harnack 
inequality theory for nonnegative solutions of the homogeneous
equations $L_\phi u=0$ in terms of the pinching of the Hessian determinant $$\lambda\leq \det D^{2} \phi\leq \Lambda.$$
Accordingly, they obtained H\"older continuity for solutions. Their approach is based on that of Krylov and Safonov \cite{KS1, KS2} on
H\"older estimates for linear, uniformly elliptic equations in general form with measurable coefficients, with sections replacing Euclidean balls. The notion of sections (or cross sections) of solutions to the Monge-Amp\`ere equation
 was first introduced and studied by Caffarelli \cite{C1,C2,C3,C4}, and plays an important role in his fundamental interior $W^{2,p}$ estimates \cite{C2}. Sections are defined 
as sublevel sets of convex solutions after subtracting their supporting hyperplanes. 
They have the same role as Euclidean balls have in the classical theory. 
A Euclidean ball of radius $r$ is a section with height $r^2/2$ of the quadratic function $|x|^2/2$ whose Hessian determinant is $1$.

This theory of Caffarelli and Guti\'errez is an affine invariant version of the classical Harnack inequality for uniformly elliptic equations with measurable coefficients. In fact, 
since the linearized Monge-Amp\`ere operator $L_\phi $ can be 
written in both divergence form and non-divergence form,  Caffarelli-Guti\'errez's H\"older regularity theorem is the affine invariant analogue of De 
Giorgi-Nash-Moser's theorem and also 
Krylov-Safonov's theorem on H\"older continuity of solutions to uniformly elliptic equations in divergence and nondivergence form, respectively.  
Caffarelli and Guti\'errez's theory has already played a crucial role in 
Trudinger and Wang's resolution \cite{TW00} of Chern's conjecture in affine geometry concerning affine maximal hypersurfaces in $\RR^3$ and in Donaldson's interior estimates for Abreu's equation 
in complex geometry \cite{D2}, paving the way for his resolution of the constant scalar curvature problem for toric surfaces \cite{D4}. It was also used by Caffarelli and Silvestre in their
pioneering paper on nonlocal equations \cite{CS} to prove several regularity results for problems involving the fractional powers of 
the Laplacian or other integro-differential equations.

The theory of Caffarelli and Guti\'errez is a basic building block in developing higher regularity results ($C^{1,\alpha}$, $W^{2, p}$ estimates) for the linearized Monge-Amp\`ere equations, both 
in the interior and at the boundary; see for example the works of Guti\'errez and Nguyen \cite{GN1, GN2}; Savin and the author \cite{LS}, Nguyen and the author \cite{LN2}.
It is worth noting that these boundary regularity estimates for the linearized Monge-Amp\`ere equations found interesting applications to nonlinear, fourth order, geometric PDEs. These
include: establishing
global second derivative estimates for the second boundary value problem of the prescribed affine maximal surface and Abreu's equations under optimal
integrability conditions on the affine mean curvature \cite{L2}, and proving global regularity for minimizers with prescribed determinant of certain functionals motivated by the Mabuchi functional in complex geometry in two dimensions \cite{LS2}. 

Recently, Maldonado \cite{M1, M} extended the Caffarelli-Guti\'errez theory and established the interior Harnack inequality for nonnegative solutions of the linearized Monge-Amp\`ere equations, with and 
without lower-order terms, under minimal geometric conditions on the Monge-Amp\`ere measure $\det D^2\phi$ of $\phi$, namely, it is 
doubling with respect to the center of mass on the sections of $\phi$.

In this paper, we establish the corresponding boundary Harnack inequality for nonnegative solutions of the linearized Monge-Amp\`ere equations under natural assumptions 
on the domain, Monge-Amp\`ere measures and boundary data. We give two applications of this boundary Harnack inequality.
In the first application, we prove a Comparison Theorem for two positive solutions
to the linearized Monge-Amp\`ere equations; here we assume furthermore that one of the two solutions vanishes on the boundary.
In the second application, we obtain a sharp upper bound for the Green's function of the linearized Monge-Amp\`ere operator and its uniform $L^p$-integrability
for the same integrability range as that of the Green's function of the Laplace operator.

The rest of the paper is organized as follows. In Section \ref{main-sec}, we state our main results (Theorems \ref{carl}, \ref{compa} and \ref{Gthm}). The main tools used in the proofs
of the main results will be recalled and established in the next two sections. In 
Section \ref{prelim_sec}, we recall basic geometric properties of solutions to the Monge-Amp\`ere equations and their rescalings
using the boundary Localization Theorem, the Localization Theorem at the boundary
for solutions to the Monge-Amp\`ere equations, 
and Caffarelli-Guti\'errez's interior Harnack inequality for solutions to the linearized Monge-Amp\`ere equations.
We establish some fine geometric properties of boundary sections of solutions to the Monge-Amp\`ere equation in Section \ref{quasi_sec}. These include Theorem \ref{pst}
and Lemma \ref{chain_lem} concerning a chain property of sections of $\phi$.
The proofs of Theorems \ref{carl} and \ref{compa} will be given in Section \ref{H_sec}. 
Theorem \ref{Gthm} will be proved in the final Section \ref{G_sec}.

\section{Statement of the main results}
 
\label{main-sec}

We assume the following global information on the convex domain $\Omega$ and the convex function $\phi$.
\begin{equation}
\Omega\subset B_{1/\rho}(0), ~\text{and for each $y\in\p\Omega$ there is a ball $B_{\rho}(z)\subset \Omega$ that is tangent to}~ \p 
\Omega~ \text{at}~ y.
\label{global-tang-int}
\end{equation}
Let $\phi : \overline \Omega \rightarrow \RR$, $\phi \in C^{0,1}(\overline 
\Omega) 
\cap 
C^2(\Omega)$  be a convex function satisfying
\begin{equation}\label{eq_u}
\det D^2 \phi =g, \quad \quad 0 <\lambda \leq g \leq \Lambda \quad \text{in $\Omega$}.
\end{equation}
Assume further that on $\p \Omega$, $\phi$ 
separates quadratically from its 
tangent planes,  namely
\begin{equation}
\label{global-sep}
 \rho\abs{x-x_{0}}^2 \leq \phi(x)- \phi(x_{0})-\nabla \phi(x_{0}) \cdot (x- x_{0})
 \leq \rho^{-1}\abs{x-x_{0}}^2, ~\forall x, x_{0}\in\p\Omega.
\end{equation}
We note that, given (\ref{eq_u}), the quadratic separation from tangent planes on the boundary (\ref{global-sep}) holds if 
$\phi\mid_{\p\Omega}, \p\Omega\in C^{3},~\text{and}~ \Omega~\text{is uniformly convex}$
(see \cite[Proposition 3.2]{S2}).

The section of $\phi$ centered at $x\in \overline \Omega$ with height $h$ is defined by
\begin{equation*}
 S_{\phi} (x, h) :=\big\{y\in \overline \Omega: \quad \phi(y) < \phi(x) + \nabla \phi(x)\cdot (y- x) +h\big\}.
\end{equation*}
For  $x\in \Omega$, we denote by $\bar{h}(x)$ the maximal height of all sections of $\phi$ centered at $x$ and contained in $\Omega$, that is,
$$\bar{h}(x): =\sup\{h\geq 0| \quad S_{\phi}(x, h)\subset \Omega\}.$$
In this case, $S_{\phi}(x, \bar{h}(x))$  is called the maximal interior section of $\phi$ with center $x\in\Omega$.

Denote by $\Phi= (\Phi^{ij})\equiv (\det D^2 \phi) (D^2 \phi)^{-1}$ the cofactor matrix of the Hessian matrix $D^2\phi= (\phi_{ij})$. $\Phi$ is divergence-free.
Then, the linearized operator of the Monge-Amp\`ere equation (\ref{eq_u}) is given by
$$L_\phi v:= -\sum_{i, j=1}^n\Phi^{ij}v_{ij}\equiv - \sum_{i, j=1}^n(\Phi^{ij} v)_{ij}.$$
We denote by $c, \bar{c}, c_1, c_2, C, \bar C, C_{1}, C_{2}, \theta_{0}, \theta_{\ast}, M, M_0, M_1, \cdots$, positive constants depending only on $\rho$, $\lambda$, $\Lambda$, 
$n$, and their values may change from line to line whenever 
there is no possibility of confusion. We refer to such constants as {\it universal constants}.

We can assume that all functions $\phi$, $u$ in this paper are smooth and thus solutions can be interpreted in the classical sense. However, our estimates do not depend on the assumed smoothness but only on the given structural constants.

Our first main result is concerned with boundary Harnack inequality for nonnegative solutions to the linearized Monge-Amp\`ere equations. 
\begin{theorem} [Boundary Harnack inequality or Carleson estimate]\label{carl} Assume that $\Omega$ and $\phi$ satisfy \eqref{global-tang-int}--\eqref{global-sep}. Let $x_0\in\p\Omega$ and $0<t\leq c$ universally small.
 Suppose that $u\geq 0$ is a continuous solution of $L_\phi u=0$ in $\Omega\cap S_\phi(x_0, t)$ with $u=0$ on $\p\Omega\cap S_\phi(x_0, t)$. Let $P_0\in\p S_\phi(x_0, t/4)$ be 
 any point satisfying $dist(P_0,\p\Omega)\geq c_1 t^{1/2}$ for some universal constant
 $c_1$. Then there is a universal constant $M_0>0$ 
 depending only on $\rho$, $\lambda$, $\Lambda$, 
$n$
 such that
 $$\max_{x\in \overline{S_\phi(x_0, t/2)}} u(x)\leq M_0 u(P_0).$$
 \end{theorem}
 \begin{remark}
 \label{cen-rem}
 The point $P_0$ in the statement of Theorem \ref{carl} always exists by Lemma \ref{cen-lem}. Also by this lemma,
 $dist(P_0,\p\Omega)$ is comparable, with universal constants,  to the largest distance from a point on $\p S_\phi(x_0, t/2)$ to 
 the boundary $\p\Omega$. By abuse of notion, we call $P_0$ the {\it apogee} of $S_\phi (x_0, t/2)$.
 \end{remark}
 Theorem \ref{carl} is a special case of the following Comparison Theorem. A particular case of this theorem asserts
that any two positive solutions of $L_\phi u=0$ in $\Omega$ which vanish on a portion of the
boundary must vanish at the same rate.
 \begin{theorem}[Comparison Theorem]\label{compa} Assume that $\Omega$ and $\phi$ satisfy \eqref{global-tang-int}--\eqref{global-sep}. Let $x_0\in\p\Omega$ and $0<t\leq c$ universally small.
 Suppose that $u, v\geq 0$ are continuous solutions of $L_\phi u=L_\phi v=0$ in $\Omega\cap S_\phi(x_0, t)$ with $u=0$ on $\p\Omega\cap S_\phi(x_0, t)$ and $v>0$ in $\Omega\cap S_\phi(x_0, t)$. Let $P\in\p S_\phi(x_0, t/4)$ be any point satisfying $dist(P,\p\Omega)\geq c_1 t^{1/2}$ for some universal constant
 $c_1$. Then there is a universal constant $M_0>0$ 
 depending only on $\rho$, $\lambda$, $\Lambda$, 
$n$
 such that
 $$\sup_{x\in S_\phi(x_0, t/2)\cap\Omega} \frac{u(x)}{v(x)}\leq M_0 \frac{u(P)}{v(P)}.$$
 \end{theorem}
The proofs of Theorems \ref{carl} and \ref{compa} will be given in Section \ref{H_sec}.

Theorem \ref{compa} is an affine invariant
 analogue of the comparison theorems in \cite{CFMS, B}. In \cite{CFMS}, Caffarelli, Fabes, Mortola and Salsa proved a comparison theorem for positive solutions of 
 linear, uniform elliptic equations in divergence form. In \cite{B}, Bauman 
 proved a comparison theorem for positive solutions of linear, uniform elliptic equations in non-divergence form with continuous coefficients. 
 \begin{remark}
  We require $t$ to be universally small in Theorems \ref{carl} and \ref{compa} for convenience. A simple covering argument with the help of the boundary Localization Theorem 
  \ref{main_loc} and the interior Harnack inequality in Theorem \ref{Holder_thm} shows that the above theorems hold for any $t$ satisfying 
  $S_\phi(x_0, t)\not\supset \overline{\Omega}$.
 \end{remark}

As an application of Theorem \ref{carl}, we obtain sharp upper bound for the Green's function of the linearized Monge-Amp\`ere operator $L_{\phi}$. 
Let $g_{V}(x, y)$ be the Green's function of $L_\phi$ in $V$ with pole $y\in V\cap\Omega$ where $V\subset\overline{\Omega}$, that is
$g_V(\cdot, y)$ is a positive solution of 
$$L_\phi g_V(\cdot, y)=\delta_{y}~\text{in} ~V\cap\Omega,~\text{and}~
 g_V(\cdot, y) =0~\text{on}~\p V.$$
\begin{theorem} [Bound on Green's function] \label{Gthm} Assume that $\Omega$ and $\phi$ satisfy \eqref{global-tang-int}--\eqref{global-sep}.
Let $x_0\in\Omega$ and $0<t_0<c(n,\lambda,\Lambda,\rho).$ If $x\in \p S_\phi (x_0, t_0)$ then
 $$g_\Omega(x, x_0)\leq \begin{cases} C(n,\lambda,\Lambda, \rho)t_0^{-\frac{n-2}{2}} &\mbox{if } n \geq 3 \\
C(n,\lambda,\Lambda, \rho)|\log t_0| & \mbox{if } n = 2. \end{cases} $$
\end{theorem}
Theorem \ref{Gthm} is a natural global counterpart of \cite[Theorem 1.1]{L} where we established sharp upper bound for the Green's function of $L_\phi$ in the interior domain
$V$ of $\Omega$ (see also \cite[Lemma 3.3]{TiW} and \cite[Theorem 3]{MRL} for related interior results).
From Theorem \ref{Gthm}, we obtain the uniform $L^p$ bound ($p<\frac{n}{n-2}$) of the Green's function.
\begin{cor} [$L^p$-integrability of Green's function]
 \label{lpbound}
  Assume that $\Omega$ and $\phi$ satisfy \eqref{global-tang-int}--\eqref{global-sep}. Let $p\in (1,\frac{n}{n-2})$ if $n\geq 3$ and $p\in (1,\infty)$ if $n=2$. Then for all $x_0\in\Omega$, we have
$$\int_{\Omega}g_{\Omega}^p (x, x_0) dx\leq C(n, \lambda,\Lambda, \rho, p).$$
More generally, for all $x_0\in\Omega$ and $t\leq 
c_1(n, \lambda,\Lambda, \rho, p)$
we have
$$\sup_{x\in S_\phi(x_0, t)\cap\Omega} \int_{S_\phi(x_0, t)} g_{S_\phi(x_0, t)}^p (y, x) dy \leq C(n, \lambda,\Lambda, \rho, p) |S_\phi(x_0, t)|^{1-\frac{n-2}{n}p}.$$
\end{cor}
The proofs of Theorem \ref{Gthm} and Corollary \ref{lpbound} will be given in Section \ref{G_sec}.
\begin{remark}
Note that, under the assumptions \eqref{global-tang-int}--\eqref{global-sep}, 
$D^2\phi$ is not bounded. In fact, the best regularity we can have from \eqref{global-tang-int}--\eqref{global-sep} only is $D^2\phi\in L^{1+\varepsilon}(\Omega)$ where 
$\varepsilon$ is a small constant depending on $n,\lambda$ and $\Lambda$. This follows from De Philippis-Figalli-Savin's interior $W^{2, 1+\varepsilon}$ estimates 
\cite{DPFS} for the Monge-Amp\`ere equation combined with Savin's techniques \cite{S3} in obtaining global regularity.

Therefore, Corollary \ref{lpbound} is interesting. It establishes the same global integrability of the Green's function for 
the linearized Monge-Amp\`ere operator as the Green's function of the Laplace operator which
corresponds to $\phi(x)=|x|^2/2$. 

Moreover, Corollary \ref{lpbound} also says that, as a degenerate and singular non-divergence form operator, $L_\phi$ has Green's function with global $L^p-$integrability
higher than that of a typical uniformly elliptic operator in non-divergence form as established by Fabes and Stroock \cite[Corollary 2.4]{FS}.
\end{remark}
Our boundary Harnack estimates in Theorems \ref{carl} and \ref{compa} depend only on the bounds on the Hessian determinant $\det D^{2} \phi$, the quadratic separation of
$\phi$ from its tangent planes on the boundary $\p\Omega$ and the geometry of $\Omega$. Under these assumptions, the linearized Monge-Amp\`ere operator $L_\phi$ is in general 
not uniformly elliptic, i.e., the eigenvalues of $\Phi = (\Phi^{ij})$ are not necessarily bounded away from $0$ and $\infty.$ Moreover, $L_\phi$ can be possibly 
singular near the boundary (see, for example \cite{LS, W}). The 
degeneracy and singularity of $L_\phi$ are the main difficulties in establishing our boundary Harnack inequalities. We handle the degeneracy of $L_\phi$ by working 
as in \cite{CG} with sections of solutions to the Monge-Amp\`ere equations. These sections have the same role as Euclidean balls have in the 
classical theory. To overcome the singularity of $L_\phi$ near the boundary, we use a Localization Theorem at the boundary for solutions to the 
Monge-Amp\`ere equations which was obtained by Savin in \cite{S1,S2}. 

Our main tools in the proofs of the main results are the boundary Localization Theorem for the Monge-Amp\`ere equations \cite{S1, S2}, the 
interior Harnack estimates for solutions to the linearized Monge-Amp\`ere equations which were established in \cite{CG} and fine geometric properties of solutions
to the Monge-Amp\`ere equation which were obtained in \cite{LN} and further  elaborated in this paper. 
Among these, we would like to point out the following geometric property of sections of $\phi$ which is a crucial ingredient in the proof of Theorem \ref{carl}. 
It says quantitatively that sections of solutions to the Monge-Amp\`ere equations share many properties as Euclidean balls.
It is a global version of \cite[Theorem 3.3.10]{G} and could be of independent interest. 
\begin{theorem}
 \label{pst}
  Assume that $\Omega$ and $\phi$ satisfy \eqref{global-tang-int}--\eqref{global-sep}. Fix a universal constant $M$ such that 
any section of $\phi$ with height $M$ contains $\overline{\Omega}$. Then, 
 there exist universal constants $c_0>0$ and $p_1\geq 1$ such that 
 \begin{myindentpar}{1cm}
 (i) if $0< r<s\leq 3, s-r\leq 1, 0<t\leq M$ and $x_1\in S_\phi(x_0, r t)$ where $x_0\in\overline{\Omega}$, then
 $$S_{\phi}(x_1, c_0 (s-r)^{p_1} t)\subset S_{\phi}(x_0, s t);$$
 (ii)  if $0< r<s< 1, 0<t\leq M$ and $x_1\in S_\phi(x_0, t)\backslash S_{\phi}(x_0, st)$ where $x_0\in\overline{\Omega}$, then
 $$S_{\phi}(x_1, c_0 (s-r)^{p_1} t)\cap S_{\phi}(x_0, rt)=\emptyset.$$
 \end{myindentpar}
\end{theorem}
For the existence of the universal constant $M$ in the statement of Theorem \ref{pst}, see Lemma \ref{global-gradient}.
The proof of Theorem \ref{pst} will be given in Section \ref{quasi_sec}. 

\begin{remark} The results of this paper, in particular,  Theorem \ref{carl} and Corollary \ref{lpbound}, have been recently used
by Nguyen and the author \cite{LN3} to establish global $W^{1,p}$ estimates for all $p<\frac{nq}{n-q}$ for solutions to 
the nonhomogeneous linearized Monge-Amp\`ere equations $L_\phi u= f$ when $f$ belongs to $L^q$ 
where $n/2<q\leq n$. 
\end{remark}

\section{Geometry of the Monge-Amp\`ere equations}
\label{prelim_sec}
In this section, we recall the main tools used in the proofs of our main results.
These include properties of sections of solutions to the Monge-Amp\`ere equation and their rescalings using the boundary Localization Theorem, 
the Localization Theorem at the boundary
for solutions to the Monge-Amp\`ere equation, 
and Caffarelli-Guti\'errez's interior Harnack inequality for solutions to the linearized Monge-Amp\`ere equation.
\subsection{Geometry of sections of solutions to the Monge-Amp\`ere equations}
In this section, we assume that
$$\lambda\leq \det D^2 \phi\leq\Lambda~\text{in}~\Omega.$$
Throughout, we use the following volume growth for compactly supported sections:
\begin{lemma} (\cite[Corollary 3.2.4]{G}) \label{vol_lem} If $S_\phi(x, t)\subset\subset \Omega$ then
$$c_1(n,\lambda,\Lambda)t^{\frac{n}{2}}\leq |S_\phi(x, t)|\leq C_1(n,\lambda,\Lambda)t^{\frac{n}{2}}.$$
\end{lemma}
We will use the following consequence of $C^{1,\alpha}$ estimates:
\begin{lemma} (\cite[Theorem 3.3.8]{G})\label{C1alpha}
 If $B(0, 1)\subset S_\phi(x, r)\subset B(0, n)$ then for any $t\leq r/2$, we have
 $$S_\phi(x, t)\subset B(x, C_1t^{\mu})$$
 for universal constants $\mu(n,\lambda,\Lambda)\in (0, 1)$ and $C_1$ depending only on $n,\lambda$, and $\Lambda$.
\end{lemma}

The Caffarelli-Guti\'errez's Harnack inequality for the linearized Monge-Amp\`ere equations states:
\begin{theorem}(\cite[Theorem 5]{CG}) \label{Holder_thm}
For each compactly supported section $S_\phi (x, t)\subset\subset\Omega$, and any nonnegative solution $v$ of $L_\phi v=0$ in $S_\phi (x, t)$, we have
$$\sup_{S_\phi(x,2\tau t)} v\leq C_1(n,\lambda,\Lambda) \inf_{S_\phi(x, 2\tau t)} v$$
for a universal $\tau\in (0, 1/2)$ depending only on $n$, $\lambda$ and $\Lambda$.
\end{theorem}

We also need the following Vitali type covering lemma.
\begin{lemma}(\cite[Lemma 2.3]{S3}) Let $D$ be a compact set in $\Omega$ and assume that for each $x\in D$ we associate a corresponding section $S_\phi(x, h)\subset\subset\Omega$. Then 
we can find a finite number of these sections $S_\phi(x_i, h_i), i=1,\cdots, m,$ such that
$$D \subset \bigcup_{i=1}^m S_\phi(x_i, h_i),~\text{with}~ S_\phi(x_i, \delta h_i)~\text{disjoint},$$
where $\delta>0$ is a small constant that depends only on $\lambda$, $\Lambda$ and $n$.
\label{cov_lem}
\end{lemma}
We recall the following global results on sections.
\begin{theorem} (Engulfing property of sections, \cite[Theorem 2.1]{LN})
\label{engulfing2}
Assume that $\Omega$ and $\phi$ satisfy \eqref{global-tang-int}--\eqref{global-sep}. Then, there exists a universal constant $\theta_{\ast}>1$ such that if $y\in S_\phi(x, t)$ with $x\in\overline{\Omega}$ and $t>0$, then $S_\phi(x, t)\subset S_\phi(y, \theta_{\ast}t).$
\end{theorem}

\begin{lemma}(\cite[Proposition 2.6]{LS}, \cite[Lemma 3.4 and Corollary 2.4]{LN})\label{global-gradient}
 Assume that $\Omega$ and $\phi$ satisfy \eqref{global-tang-int}--\eqref{global-sep}. Then, there are universal constants $\alpha\in (0,1)$, $C_{\alpha}, c_0, C_1$ and $M$ 
such that
\begin{myindentpar}{1cm}
(i) 
$$
[\nabla\phi]_{C^{\alpha}(\overline{\Omega})}\leq C_{\alpha};
$$
(ii) 
$$
 S_\phi(x,M)\supset \overline{\Omega} \quad\mbox{for all } x\in \overline{\Omega};
$$
(iii) for any section $S_{\phi}(x, t)$ with $x\in \overline{\Omega}$  and $t\leq c_0$, we have
\begin{equation*}
C_{1}^{-1} t^{n/2}\leq |S_{\phi}(x, t)|\leq C_{1} t^{n/2}.
\end{equation*}
\end{myindentpar}
\end{lemma}

\subsection{The Localization Theorem and properties of the rescaled functions} 
In this section, we assume that 
$\Omega$ and $\phi$ satisfy \eqref{global-tang-int}--\eqref{global-sep}.
We now focus on sections centered at a  point on the boundary $\p\Omega$ and describe their geometry. Assume this boundary point to be $0$ and by
 \eqref{global-tang-int}, we can also assume that
\begin{equation}\label{om_ass}
B_\rho(\rho e_n) \subset \, \Omega \, \subset \{x_n \geq 0\} \cap B_{\frac 1\rho},
\end{equation}
where $\rho>0$ is the constant given by condition \eqref{global-tang-int}. 
After subtracting a linear function, we can assume further that
\begin{equation}\label{0grad}
\phi(0)=0, \quad \nabla \phi(0)=0.
\end{equation}
By (\ref{global-sep}), the boundary data $\phi$ has quadratic growth near the hyperplane $\{x_n=0\}$. Hence, as $h \rightarrow 0$, $S_\phi(0, h)$ is equivalent to a half-ellipsoid 
centered at 0. This follows from
the Localization Theorem proved by Savin in \cite{S1,S2}. Precisely, this theorem reads as follows.
\begin{theorem}[Localization Theorem \cite{S1,S2}]\label{main_loc} Assume that $\Omega$ satisfies \eqref{om_ass} and $\phi$ satisfies 
\eqref{eq_u}, \eqref{0grad}, and
\begin{equation*}\label{commentstar}\rho |x|^2 \leq \phi(x) \leq \rho^{-1} 
|x|^2 \quad \text{on $\p \Omega \cap \{x_n \leq \rho\}.$}
\end{equation*}
Then, for each $h\leq k$ there exists an ellipsoid $E_h$ of volume $\omega_{n}h^{n/2}$ 
such that
\[
 kE_h \cap \overline \Omega \, \subset \, S_\phi(0, h) \, \subset \, k^{-1}E_h \cap \overline \Omega.
\]
Moreover, the ellipsoid $E_h$ is obtained from the ball of radius $h^{1/2}$ by a
linear transformation $A_h^{-1}$ (sliding along the $x_n=0$ plane)
$$A_hE_h= h^{1/2}B_1,\quad \det A_{h} =1,$$
$$A_h(x) = x - \tau_h x_n, \quad \tau_h = (\tau_1, \tau_2, \ldots, 
\tau_{n-1}, 0), $$
with
$ |\tau_{h}| \leq k^{-1} |\log h|$.
The constant $k$ above depends only on $\rho, \lambda, \Lambda, n$.
\end{theorem}

Throughout, we denote by $B_r= B_r(0)$ the Euclidean ball centered at $0$ with radius $r$ and $B_r^{+}= B_r \cap\{x\in\RR^n: x_n\geq 0\}$.

Let $\phi$ and $\Omega$ satisfy the hypotheses of the Localization Theorem \ref{main_loc} at the 
origin. We know that for all $h \le k$,
 $S_\phi(0, h)$ satisfies 
$$k E_h \cap \bar \Omega  \subset S_\phi(0, h) \subset k^{-1} E_h,$$ 
with $A_h$ being a linear transformation 
and
$$\det A_{h} = 1,\quad E_h=A^{-1}_hB_{h^{1/2}},  \quad A_hx=x-\tau_hx_n$$
$$\tau_h \cdot e_n=0, \quad \|A_h^{-1}\|, \,\|A_h\| \le k^{-1} |\log h|.$$
This gives 
\begin{equation}\label{small-sec}
 \overline \Omega \cap B^{+}_{c_2h^{1/2}/|\log h|}\subset S_\phi(0, h) 
\subset B^{+}_{C_2 h^{1/2} |\log h|}.
\end{equation}
We denote the rescaled functions by \begin{equation*}
 \phi_h(x):=\frac{\phi(h^{1/2}A^{-1}_hx)}{h}.
\end{equation*}
The function $\phi_h$ is continuous and is defined in $\overline \Omega_h$ with
 $\Omega_h:= h^{-1/2}A_h \Omega,$
 and solves 
the Monge-Amp\`ere equation
 $$\det D^2 \phi_h=g_h(x), \quad \quad \lambda \le g_h(x) :=g(h^{1/2}A_h^{-1}x) \le \Lambda.$$
 The section at height 1 for $\phi_h$ centered at the origin satisfies
$S_{\phi_{h}}(0, 1)=h^{-1/2}A_hS_\phi (0, h),$ and by Theorem \ref{main_loc}, we 
obtain $$B_k \cap \overline \Omega_h \subset S_{\phi_{h}}(0, 1) \subset B_{k^{-1}}^+.$$

Some properties of the rescaled function $\phi_h$ was established in \cite[Lemma 4.2 and Lemma 5.4]{LS}. For later use, we record them here.
\begin{lemma}If $h\leq c$, then 
\begin{myindentpar}{1cm}
(i)
$\phi_h$ and $S_{\phi_h}(0, 1)$ satisfy the hypotheses of the Localization Theorem \ref{main_loc} at all $x_0\in\p\Omega_h\cap B(0, c)$,  for a small constant $\tilde \rho$ depending only on 
$\rho, n,\lambda,\Lambda$;\\
(ii) if $r \le c$ small, we have $$|\nabla \phi_h| \le 
C r |\log r|^2 \quad \mbox{in} \quad \overline \Omega_h \cap B_r;$$
(iii) $\p \Omega_h \cap B_{2/k}$ is a graph 
in the $e_n$ direction whose $C^{1,1}$ norm 
is bounded by $C h^{1/2}$;\\
(iv) If $\delta$ is universally small, then for any $x\in S_{\phi_h}(0,\delta)\cap\Omega_h$,  the maximal interior section $S_{\phi_h}(x, \bar h(x))$ of $\phi_h$ centered at $x$ is tangent to $\p\Omega_h$ at $z\in \p\Omega_h\cap B_{c/2}$.
\end{myindentpar}
\label{sep-lem}
\end{lemma}

\section{Geometric properties of boundary sections}
\label{quasi_sec}
In this section, we prove Theorem \ref{pst}
and establish a chain property for sections of solutions to the Monge-Amp\`ere equations in Lemma \ref{chain_lem}.

For the proof of Theorem \ref{pst}, we need two additional results: Proposition \ref{tan_sec} and Lemma \ref{size_sec}. Proposition \ref{tan_sec} is concerned with the shape of maximal interior sections. Lemma \ref{size_sec} estimates the size of a section of $\phi$ in terms of its height. 
\begin{propo}\label{tan_sec}
Let $\phi$ and $\Omega$ satisfy the hypotheses of the Localization Theorem \ref{main_loc} at the 
origin. Assume that for some $y \in \Omega$ the maximal interior section $S_{\phi}(y, \bar h(y)) \subset \Omega$
is tangent to $\p \Omega$ at $0$, that is, $\p S_{\phi}(y, \bar h(y))\cap \p\Omega =\{0\}$. Then, if  $h:=\bar h(y) \le c$ with $c$ universal, 
\begin{myindentpar}{1cm}
(i) there exists a small 
 positive constant $k_0<k$ depending 
only on $\rho $, $\lambda$, $\Lambda$  and $n$ such that
$$ \nabla \phi(y)=a e_n 
\quad \mbox{for some} \quad   a \in [k_0 h^{1/2}, k_0^{-1} h^{1/2}],$$
$$k_0 E_h \subset S_{\phi}(y, h) -y\subset k_0^{-1} E_h, \quad \quad k_0 h^{1/2} \le dist(y,\p \Omega) \le k_0^{-1} h^{1/2}, \quad $$
with $E_h$ and $k$ the ellipsoid and constant defined in the Localization Theorem \ref{main_loc};\\
(ii) if $h/2<t\leq c$ then $S_\phi(y, 2t)\subset S_\phi(0, (2k_0^{-4} + 4) t)$;\\
(iii) If $t\leq h/2$ then 
$S(y, t)\subset B(y, C_4 t^{\bar\mu})$
for universal constants
$\bar\mu(n,\lambda,\Lambda)\in (0, 1/8)$ and $C_4$.
\end{myindentpar}
\end{propo}

\begin{lemma}\label{size_sec}  Assume that $\Omega$ and $\phi$ satisfy \eqref{global-tang-int}--\eqref{global-sep}. Let $M$ be as in Lemma \ref{global-gradient} and 
let $\bar\mu$ be as in Proposition \ref{tan_sec}. Then, there is a universal constant $\bar C>0$ such that
 for all $t\leq M$ and $y\in\overline{\Omega}$,
 $$S_\phi(y, t)\subset B(y, \bar C t^{\bar\mu}).$$
\end{lemma}

\begin{proof}[Proof of Proposition \ref{tan_sec}]  (i) is a consequence of Theorem \ref{main_loc} and was proved in \cite{S3} (see also \cite[Proposition 3.2]{LS}).
 (ii) comes from the proof of \cite[Proposition 2.3]{LN} (see equation (4.11) there).

We now prove (iii). By (i),
$S_\phi(y,h)$ is equivalent to an ellipsoid $E_h$, i.e.,
$$k_0 E_h \subset S_\phi(y,  h)-y \subset k_0^{-1}E_h,$$
where
\begin{equation}E_h:= h^{1/2}A_{h}^{-1}B_1, \quad \mbox{with} \quad \|A_{h}\|, \|A_{h}^{-1} \| \le C |\log h|; \det A_{h}=1.
\label{eh}
\end{equation}
The rescaling $\tilde \phi$ of $\phi$ 
$$\tilde \phi(\tilde x):=\frac {1}{ h} \left[\phi(y+ h^{1/2}A_{h}^{-1}\tilde x)-\phi (y) -\nabla \phi(y) \cdot (h^{1/2}A_{h}^{-1}\tilde x)\right]$$
satisfies
$$\det D^2\tilde \phi(\tilde x)=\tilde g(\tilde x):=g(y+ h^{1/2}A_{h}^{-1}\tilde x)\in [\lambda,\Lambda],  $$
and
\begin{equation}
\label{normalsect}
B_{k_0} \subset  S_{\tilde\phi}(0, 1) \subset B_{k_0^{-1}}, \quad \quad  S_{\tilde\phi}(0, 1)= h^{-1/2} A_{ h}(S_\phi(y, h)- y),
\end{equation}
where we recall that $ S_{\tilde\phi}(0, 1)$ represents the section of $\tilde \phi$ at the origin with height 1.

With (\ref{normalsect}), we apply Lemma \ref{C1alpha} to obtain 
$\mu(n,\lambda,\Lambda)\in (0, 1)$ and $C_4(n,\lambda,\Lambda,\rho)$ such that
$$h^{-1/2} A_{ h}(S_\phi(y, t)- y)=S_{\tilde\phi} (0, \frac{t}{ h})\subset B(0, C_4 (\frac{t}{h})^{\mu}).$$
Using (\ref{eh}), $t/h\leq \frac{1}{2}$ and $ h \leq c$, we can take $\bar\mu= \min\{\mu,\frac{1}{10}\}$ and obtain
$$S_\phi(y, t)- y\subset h^{1/2}A_h^{-1} (B(0, C_4 (\frac{t}{h})^{\mu})) \subset h^{1/4} B(0, C_4 (\frac{t}{h})^{\bar \mu}) \subset B(0, C_4 t^{\bar\mu}).$$
Hence (iii) is proved.
\end{proof}
 \begin{proof} [Proof of Lemma \ref{size_sec}] In this proof, we drop the dependence on $\phi$ of sections when no confusion arises.
We first prove the lemma for the case where $y$ is a boundary point which can be assumed to be 
 $0\in\p\Omega$. By rotating coordinates, and subtracting a linear function from $\phi$, we can assume that
 $\phi$ and $\Omega$ satisfy the hypotheses of the Localization Theorem \ref{main_loc} at the 
origin. Thus, if $t\leq k(\leq 1)$, then from (\ref{small-sec}), we have
 $$S(0, t)\subset B(0, C_2 t^{1/2}|\text{log } t|)\subset B(0, C_3 t^{1/4})$$
 for some $C_3$ universal, depending only on $\rho,n,\lambda$ and $\Lambda$.
 By increasing $C_3$ if necessary, we find that
 \begin{equation}S(0, t)\subset B(0, C_3 t^{\bar\mu})~\text{for all } t\leq M.
  \label{ybdr}
 \end{equation}

Next, we prove the lemma for $y\in\Omega$ away from the boundary, that is 
$r:=dist (y,\partial\Omega) \ge c,$ for $c$ universal. In this case, we can use Lemma \ref{C1alpha} and the strict convexity of $\phi$
 which follows from the boundary $C^{1,\alpha}$ regularity of $\phi$ for all $\alpha\in (0,1)$ as observed in \cite[Lemma 4.1]{LS}
 and Caffarelli's Localization Theorem \cite{C1}. We then find a $\mu\in (0, 1)$ depending only on
$n,\lambda,\Lambda$ and $ C_0$ depending on $\rho,n,\lambda,\Lambda$ such that
$S_\phi(y, t)\subset B(y, C_0 t^{\mu})~\text{for all } t\leq M.$
By increasing $C_0$ if necessary, we find that
\begin{equation}S_\phi(y, t)\subset B(y, C_0 t^{\bar\mu})~\text{for all } t\leq M.
 \label{yfbdr}
\end{equation}
Finally, we prove the lemma for $y\in\Omega$ near the boundary $\p\Omega$.
Let $y\in \Omega $ with $r:=dist (y,\partial\Omega) \le c,$ for $c$ universal. Consider the section $S(y, t)$ with $t\leq M$. Then, either there exists $z\in\partial\Omega$ such that $z\in S(y, 2t)$
or $S(y, 2t)\subset \Omega$.
In the first case, by Theorem \ref{engulfing2}, we have
 $ S(y, 2t)\subset S(z, 2\theta_{\ast}t).$ Thus, 
as in (\ref{ybdr}), we find  
 $$S(y, t)\subset S(z,2\theta_{\ast} t)\subset B(z, C_3 t^{\bar\mu}).$$
 It follows that
 $|y-z|\leq C_3 t^{\bar\mu}$
 and therefore
 \begin{equation}S(y, t)\subset B(y, 2C_3 t^{\bar\mu})~\text{for all } t\leq M.
  \label{ynbdr2}
 \end{equation}
 In the remaining case $S(y, 2t)\subset \Omega$, we obtain from Proposition \ref{tan_sec} (iii),
\begin{equation}S(y, t)\subset B(y, C_4 t^{\bar\mu}).
 \label{ynbdr1}
\end{equation}
The lemma now follows from (\ref{ybdr})-(\ref{ynbdr1}) by choosing $\bar C=C_0 + 2C_3 + C_4.$
 \end{proof}

\begin{proof}[Proof of Theorem \ref{pst}] We prove that the conclusions of the theorem hold for $p_1=\bar\mu^{-1}$ where $0<\bar\mu (n,\lambda,\Lambda)<1/8$ is the 
universal constant in Lemma \ref{size_sec}. \\
 (i) Let $0< r<s\leq 3, s-r\leq 1$ and $0<t\leq M$ where $M$ is as in Lemma \ref{global-gradient}. Let $c, \delta$ be as in Lemma \ref{sep-lem} and let $\theta_{\ast}$ be as in Theorem \ref{engulfing2}. We first consider the case $t\leq \frac{c\delta}{4\theta_\ast}$. If 
 $S_\phi(x_0, 4t)\subset\Omega$ then the theorem follows from the interior result established in \cite[Theorem 3.3.10]{G}. Suppose now $S_\phi(x_0, 4t)\cap\p\Omega\neq\emptyset.$ Without loss of generality, we can assume that $0\in S_\phi(x_0, 4t)\cap\p\Omega$ and that $\phi$ and $\Omega$ satisfy the hypotheses
of the Localization Theorem \ref{main_loc} at $0$. By Theorem \ref{engulfing2}, we have
\begin{equation}S_\phi(x_0, 4t)\subset S_\phi(0, 4\theta_\ast t).
\label{eng0}
\end{equation}
 We use the Localization Theorem \ref{main_loc} at the origin with height
$h= 4\theta_\ast \delta^{-1}t\leq c$ and consider the rescaled function
$$\phi_h(x)= \frac{\phi(h^{1/2}A_h^{-1}x)}{h}~\text{where}~x\in\Omega_h= h^{-1/2}A_h \Omega.$$
Denote
$x_{i, h}= h^{-1/2}A_h x_i$ for $i=0, 1$. Let $\bar t= \frac{\delta}{4\theta_\ast}$.
Then, since $x_1\in S_\phi(x_0, rt)$, we have
$x_{1, h}\in S_{\phi_h}(x_{0, h}, r\bar t)$ and 
$$h^{-1/2}A_hS_{\phi}(x_1, c_0 (s-r)^{p_1} t)=S_{\phi_h}(x_{1, h}, c_0 (s-r)^{p_1}\bar t),~ h^{-1/2}A_hS_{\phi}(x_0, st)= S_{\phi_h}(x_{0, h}, s\bar t).$$
We will show that for some universally small $c_0>0$ and $p_1\geq 1$
\begin{equation}S_{\phi_h}(x_{1, h}, c_0 (s-r)^{p_1}\bar t)\subset S_{\phi_h}(x_{0, h}, s\bar t).
\label{res_inc}
\end{equation}
Suppose that $y\in S_{\phi_h}(x_{1, h}, c_0 (s- r)^{p_1}\bar t)$ and $x_{1,h}\in S_{\phi_h}(x_{0, h}, r \bar t)$. Then
\begin{eqnarray}
  \phi_h(y)&<&\phi_h(x_{1, h})+ \nabla\phi_h(x_{1, h})(y-x_{1, h}) + c_0 (s-r)^{p_1}\bar t\nonumber\\
  &<& \phi_h(x_{0, h}) +\nabla\phi_h(x_{0, h})(x_{1, h}- x_{0, h}) + r\bar t + \nabla\phi_h(x_{1, h})(y-x_{1, h}) + c_0 (s-r)^{p_1}\bar t \nonumber\\
  &=& \phi_h(x_{0, h}) + \nabla\phi_h(x_{0, h})(y-x_{0, h}) + [\nabla\phi_h(x_{1, h})-\nabla\phi_h(x_{0, h})](y-x_{1, h}) + c_0 (s-r)^{p_1}\bar t + r\bar t.
  \label{phihy}
 \end{eqnarray}
We note from (\ref{eng0}) that $x_{1, h}\in S_{\phi_h}(x_{0, h}, r\bar t)\subset S_{\phi_h}(0,\delta)$.
 By Lemma \ref{sep-lem} (ii), 
\begin{equation}|\nabla\phi_h (x_{0, h})|\leq  C ~\text{and}~ |\nabla\phi_h (x_{1, h})|\leq  C.
\label{xh_size}
\end{equation}
On the other hand, from  $y\in S_{\phi_h}(x_{1, h}, c_0 (s- r)^{p_1}\bar t)$, we can estimate for some universal constant $C$
\begin{equation}|y-x_{1, h}|\leq  C (c_0 (s- r)^{p_1} \bar t)^{\bar\mu}
\label{yxh}
\end{equation}
where we call that $\bar\mu (n,\lambda,\Lambda)\in (0, 1/8)$ is also the constant in Proposition \ref{tan_sec}. 

Indeed, if $x_{1, h}\in\p\Omega_h$, then from Lemma \ref{sep-lem}(i), $\phi_h$ and $S_{\phi_h}(0, 1)$ satisfy the hypotheses of the Localization Theorem \ref{main_loc} at
$x_{1, h}$ and hence (\ref{yxh}) follows from the estimate (\ref{small-sec}).

Consider now the case $x_{1, h}\in \Omega_h$.
By Lemma \ref{sep-lem} (iv), $ S_{\phi_h}(x_{1, h}, \bar h(x_{1, h}))$, the maximal interior section of $\phi_h$ centered at $x_{1, h}$ is 
tangent to $\p\Omega_h$ at $z\in \p\Omega_h\cap B_{c/2}$. Thus, by Lemma \ref{sep-lem} (i), the Localization Theorem \ref{main_loc} is applied at $z$
and we can apply Proposition \ref{tan_sec} to the tangent point $z$.

If $c_0(s- r)^{p_1}\bar t > \frac{1}{2}\bar h(x_{1, h})$ then from Proposition \ref{tan_sec} (ii), we have
$$S_{\phi_h}(x_{1, h}, 2c_0 (s- r)^{p_1}\bar t)\subset S_{\phi_h}(z, \bar c c_0 (s- r)^{p_1}\bar t), ~\bar c=2k_0^{-4} + 4.$$
It follows from (\ref{small-sec}) that
\begin{equation}|y-x_{1, h}|\leq diam (S_{\phi_h}(z, \bar c c_0 (s- r)^{p_1}\bar t))\leq C(\bar cc_0 (s- r)^{p_1}\bar t)^{1/3}.
 \label{yxh1}
\end{equation}

Now, it remains to consider the case $c_0(s- r)^{p_1}\bar t \leq \frac{1}{2}\bar h(x_{1, h})$. Then, by Proposition \ref{tan_sec} (iii),
\begin{equation}
 |y-x_{1, h}| \leq C_4 (c_0(s- r)^{p_1}\bar t)^{\bar\mu}.
 \label{yxh2}
\end{equation}
Combining (\ref{yxh1}) and (\ref{yxh2}), we obtain (\ref{yxh}), for the case $x_{1,h}\in\Omega_h$, as asserted.\\
It follows from (\ref{phihy})-(\ref{yxh}) and $0<s-r\leq 1$ that 
$$\phi_h(y) -\phi_h(x_{0, h}) - \nabla\phi_h(x_{0, h})(y-x_{0, h})<C (c_0 (s- r)^{p_1} \bar t)^{\bar\mu}+c_0 (s- r)^{p_1}\bar t + r\bar t< s\bar t $$
if we choose $c_0$ universally small and $p_1= \bar\mu^{-1}.$ Therefore  $S_{\phi_h}(x_{1, h}, c_0 (s-r)^{p_1}\bar t)\subset S_{\phi_h}(x_{0, h},s\bar t)$, proving (\ref{res_inc})
as claimed.

 Finally, we consider the case $M\geq t\geq\frac{c\delta}{4\theta_\ast}$.
  Suppose that $z\in S_\phi(x_1, c_0 (s-r)^{p_1}t)$ and $x_1\in S_\phi(x_0, r t)$. Then, as in (\ref{phihy}), we have
 \begin{eqnarray*}
  \phi(z)<\phi(x_0) + \nabla\phi(x_0)(z-x_0) + [\nabla\phi(x_1)-\nabla\phi(x_0)](z-x_1) + c_0 (s-r)^{p_1}t + rt.
 \end{eqnarray*}
 We show that if $c_0$ is universally small and $p_1= \bar\mu^{-1}$ as above then
 \begin{equation}
  \phi(z)-\phi(x_0) - \nabla\phi(x_0)(z-x_0) < [\nabla\phi(x_1)-\nabla\phi(x_0)](z-x_1) + c_0 (s-r)^{p_1}t + rt<st,
  \label{z-inc}
 \end{equation}
which implies that $S_{\phi}(x_1, c_0 (s-r)^{p_1}t)\subset S_{\phi}(x_0,st)$ as asserted in the Theorem.
Indeed, by Lemma \ref{size_sec}, we have
$$|z-x_1|\leq \bar C (c_0 (s-r)^{p_1} t)^{\bar \mu},~\text{and}~|x_1-x_0|\leq \bar C(rt)^{\bar \mu}.$$
It follows from the global $C^{1,\alpha}$ estimate in Lemma \ref{global-gradient} that
$$|[\nabla\phi(x_1)-\nabla\phi(x_0)](z-x_1)|\leq C_{\alpha} |x_1-x_0|^{\alpha} |z-x_1|\leq C (rt)^{\alpha\bar \mu}(c_0 (s-r)^{p_1} t)^{\bar\mu}.$$
If $\frac{c\delta}{4\theta_\ast}\leq t\leq M$ then (\ref{z-inc}) easily follows from 
$$C (rt)^{\alpha\bar\mu}(c_0 (s-r)^{p_1} t)^{\bar\mu} +c_0 (s-r)^{p_1}t + rt<st $$
if we choose $c_0$ universally small and $p_1=\bar \mu^{-1}.$\\
(ii) The proof, based on convexity, is quite similar to that of (i). We include here for example the proof for the case
$M\geq t\geq\frac{c\delta}{4\theta_\ast}$. From convexity and $x_1\in S_\phi(x_0, t)\backslash S_{\phi}(x_0, st)$, we have for all $z\in\overline{\Omega}$
 \begin{eqnarray*}
  \phi(z)&\geq &\phi(x_1)+ \nabla\phi(x_1)(z-x_1)\\
  &\geq& \phi(x_0) +\nabla\phi(x_0)(x_1- x_0) + st+ \nabla\phi(x_1)(z-x_1)  \\
  &=& \phi(x_0) + \nabla\phi(x_0)(z-x_0) + [\nabla\phi(x_1)-\nabla\phi(x_0)](z-x_1) + st.
 \end{eqnarray*}
 We will show that, for all $z\in S(x_1, c_0 (s-r)^{p_1}t)$, we have
 \begin{equation}\phi(z)- \phi(x_0) - \nabla\phi(x_0)(z-x_0) > [\nabla\phi(x_1)-\nabla\phi(x_0)](z-x_1) + st>rt,
  \label{z-out}
 \end{equation}
which implies that $S_{\phi}(x_1, c_0 (s-r)^{p_1}t)\cap S_{\phi}(x_0,rt)=\emptyset.$
 
 Suppose now $z\in S(x_1, c_0 (s-r)^{p_1}t)$. 
By Lemma \ref{size_sec}, we have
$$|z-x_1|\leq \bar C (c_0 (s-r)^{p_1} t)^{\bar\mu}~\text{and}~|x_1-x_0|\leq \bar Ct^{\bar\mu}.$$
It follows from the global $C^{1,\alpha}$ estimate in Lemma \ref{global-gradient} that
$$|[\nabla\phi(x_1)-\nabla\phi(x_0)](z-x_1)|\leq C_{\alpha} |x_1-x_0|^{\alpha} |z-x_1|\leq C t^{\alpha\bar\mu}(c_0 (s-r)^{p_1} t)^{\bar\mu}$$
and thus
$$\phi(z)-\phi(x_0) - \nabla\phi(x_0)(z-x_0)\geq st-C t^{\alpha\bar\mu}(c_0 (s-r)^{p_1} t)^{\bar\mu}\geq rt,$$
proving (\ref{z-out}),
provided that 
$$C t^{\alpha\bar\mu}(c_0 (s-r)^{p_1} t)^{\bar\mu}<(s-r)t. $$
Since $M\geq t\geq\frac{c\delta}{4\theta_\ast}$ and $0<s-r\leq 1$, this can be achieved by choosing $p_1=\bar \mu^{-1}$ and $c_0$ small. 
\end{proof}

We now establish a chain property for sections of solutions to the Monge-Amp\`ere equations.
\begin{lemma} 
\label{chain_lem}
Assume $\phi$ and $\Omega$ satisfy the hypotheses of 
the Localization Theorem \ref{main_loc} at $z$. Let $\tau=\tau(n,\lambda,\Lambda)$ be as in Theorem \ref{Holder_thm}.
Assume that $x\in\Omega$ with $dist(x,\p\Omega)\leq c_0$ universally small and satisfies $\p\Omega \cap \p S_\phi(x,\bar h(x))=\{z\}$. 
Then  
\begin{myindentpar}{1cm} 
(i) we can find a sequence $x_0=x, x_1, \cdots, x_m$ in $S_\phi(z,\bar K dist^2(x,\p\Omega))$ with $m\leq K$ such that
$$x_i\in S_{\phi}(x_{i+1},\tau\bar h(x_{i+1}))~\text{for all}~i=0, 1,\cdots, m-1$$
and
$$dist (x_m,\p\Omega)\geq 2k_0^{-2} dist (x,\p\Omega)$$
for some universal constants $K, \bar K>0$, and $k_0$ is the small constant in Proposition \ref{tan_sec}.\\
(ii) Conversely, for any universal constants $K_1$ and $K_2$ and any $y\in S_\phi(z, K_1 dist^2(x,\p\Omega))$ with $dist (y,\p\Omega)\geq K_2 dist (x,\p\Omega)$,  we can 
find a sequence $x_0=x, x_1, \cdots, x_m=y$ in the section  $ S_\phi(z,K_1 dist^2(x,\p\Omega))$ with $m$ universally bounded and
$$x_i\in S_{\phi}(x_{i+1},\tau\bar h(x_{i+1}))~\text{for all}~i=0, 1,\cdots, m-1.$$
\end{myindentpar}
\end{lemma}
\begin{proof} [Proof of Lemma \ref{chain_lem}] We give here the proof of (i) for $k_0$ as in Proposition \ref{tan_sec} while that of 
(ii) follows similarly. Assume $z=0$ for simplicity.
By Proposition \ref{tan_sec},
$$k_0 \bar h^{1/2}(x)\leq dist (x,\p\Omega)\leq k_0^{-1} \bar h^{1/2}(x).$$
Let $c$ be as in Proposition \ref{tan_sec}.
We first prove the lemma for $c_0 c \leq dist(x,\p\Omega)\leq c_0.$ 
 By (\ref{global-tang-int}) and the first inclusion in (\ref{small-sec}), we can find $y\in S_{\phi}(0, c)$ such that
$$dist(y,\p\Omega)\geq c_3(\rho,n,\lambda,\Lambda).$$
Thus if $c_0$ is universally small then for $K_1:= 6k_0^{-2}$, we have
\begin{equation}dist(y,\p\Omega) \geq K_1 dist(x,\p\Omega).
 \label{yxdist}
\end{equation}
Let
$$D= \{y\in \overline{S_\phi(0, c)}: dist(y,\p\Omega)\geq dist (x,\p\Omega)\}\subset S_\phi(0, \frac{\bar K}{4} dist^2(x,\p\Omega))$$
where $\bar K= 4 (cc_0^2)^{-1}$ is a universal constant. 
Since
$D\subset \bigcup_{y\in D}S_\phi(y, \tau\bar h(y)),$
by Lemma \ref{cov_lem}, we can find a covering
\begin{equation}D\subset \bigcup_{i=1}^{K}S_\phi(y_i, \tau\bar h(y_i))
 \label{D_cov}
\end{equation}
such that $\{S_\phi(y_i, \delta\tau\bar h(y_i)\}$ is a disjoint family of sections. 
Hence, with $y$ as in (\ref{yxdist}), we can find a sequence $x_0=x, x_1, \cdots, x_m$ in $S_\phi(0,\frac{\bar K}{4} dist^2(x,\p\Omega))$ with $m\leq K$ such that
$x_i\in S_{\phi}(x_{i+1},\tau\bar h(x_{i+1}))$ for all $i=0, 1,\cdots, m-1$.
The conclusion of the lemma follows once we prove that $K$ is universally bounded. However,
this follows from the volume estimate.
Indeed, we first note that
$|D|\leq C(n,\rho,\lambda,\Lambda).$
On the other hand, from Proposition \ref{tan_sec} (i), we have for each $i=1,\cdots, K$,
$$\bar h(y_i)\geq k_0^2 dist^2 (y_i,\p\Omega)\geq k_0^2 dist^2 (x,\p\Omega)\geq (k_0 cc_0)^2 .$$
Hence, by (\ref{D_cov}) and the lower bound on volume of sections in Lemma \ref{vol_lem},
we deduce that $K$ is universally bounded.

We now prove the lemma for $dist(x,\p\Omega)\leq cc_0$.  
Let 
$h= dist^2(x,\p\Omega)/(c_0^2 c)\leq c.$
We use the Localization Theorem \ref{main_loc} at the origin and consider the rescaled function
$$\phi_h(x)= \frac{\phi(h^{1/2}A_h^{-1}x)}{h}~\text{where}~x\in\Omega_h= h^{-1/2}A_h \Omega.$$
By Lemma \ref{sep-lem}, $\phi_h$ and  $S_{\phi_h}(0, 1)$ satisfy the hypotheses of the Localization Theorem \ref{main_loc}
at all points $x_0\in\p\Omega_h\cap B(0, c)$. 
By \cite[Proposition 4.3]{LS}, we have for all $y_h=h^{-1/2}A_h y\in\Omega_h\cap B_{k^{-1}}^+$,
\begin{equation}1-Ch^{1/2}|\text{log } h|^2 \leq \frac{h^{-1/2}dist(y,\p\Omega)}{dist(y_h,\p\Omega_h)}\leq 1+Ch^{1/2}|\text{log } h|^2. 
 \label{dist_h}
\end{equation}
If $h\leq c$ universally small, we have 
\begin{equation}Ch^{1/2}|\text{log }h|^2\leq 1/2.
 \label{h12}
\end{equation}
 This in particular implies that for $x_h= h^{-1/2}A_h x$,
$$dist(x_h,\p\Omega_h)\in [\frac{1}{2}h^{-1/2}dist(x,\p\Omega), \frac{3}{2}h^{-1/2}dist(x,\p\Omega)]= [\frac{1}{2}c_0 c^{1/2},\frac{3}{2}c_0 c^{1/2}].$$
As in the case $c_0 c \leq dist(x,\p\Omega)\leq c_0$, we can find a sequence 
$$\{x_{0, h}=x_h, x_{1, h}, \cdots, x_{m, h}\}\subset S_{\phi_h} (0, \frac{\bar K}{4} dist^2(x_h,\p\Omega_h))$$ with $m\leq K$ universal such that
$$x_{i, h}\in S_{\phi_h}(x_{i+1, h},\tau\bar h(x_{i+1, h}))~\text{for all}~i=0, 1,\cdots, m-1,$$
and
$$dist (x_{m, h},\p\Omega_h)\geq K_1 dist (x_h,\p\Omega_h).$$
It follows that for $x_i= h^{1/2}A_h^{-1} x_{i, h}$, we have $x_i\in S_{\phi}(x_{i+1},\tau\bar h(x_{i+1}))$ for all $i=0, 1,\cdots, m-1$.
Recalling (\ref{dist_h}) and (\ref{h12}), we find
$$ dist (x_{m},\p\Omega)\geq \frac{K_1}{3} dist (x,\p\Omega)= 2k_0^{-2}dist (x,\p\Omega).$$
Furthermore, for all $i=0, 1,\cdots, m$, we have from (\ref{dist_h}) and (\ref{h12})
$$x_i\in S_\phi(0, h (\bar K/4) dist^2(x_h,\p\Omega_h))\subset  S_\phi(0, \bar K dist^2(x,\p\Omega)). $$
The proof of the lemma is complete.
\end{proof}
We end this section with the following lemma which justifies the definition of the {\it apogee} of sections in Remark \ref{cen-rem} and its use in this paper. 
\begin{lemma}\label{cen-lem}
Assume that $\Omega$ and $\phi$ satisfy \eqref{global-tang-int}--\eqref{global-sep}.  Suppose that $x\in\p\Omega$ and $t\leq c_1$ universally small.
Then, for some universal constant $C>0$,
\begin{myindentpar}{1cm}
(i) $dist(y, \p\Omega)\leq C t^{1/2}$ for all $y\in S_\phi(x, t)$;\\
(ii) there exists $y\in\p S_\phi(x, t)\cap\Omega$ such that $dist(y, \p\Omega)\geq C^{-1} t^{1/2}$.
\end{myindentpar}
\end{lemma}

\begin{proof}[Proof of Lemma \ref{cen-lem}] Without loss of generality, we assume that $x=0$ and $\phi$ and $\Omega$ satisfy the hypotheses of the Localization Theorem \ref{main_loc} at the 
origin. Let $k$ be as in Theorem \ref{main_loc}. Let $c$ be as in Lemma \ref{sep-lem} and let $\theta$ be universally large 
such that $$k^{-1} (\frac{1}{\theta})^{1/2} \log \frac{1}{\theta}\leq c.$$ Suppose $t\leq c_1:= \frac{c}{\theta}$. 
Let $h=\theta t\leq c$.
We use the Localization Theorem \ref{main_loc} at the origin and consider the rescaled function
$$\phi_h(x)= \frac{\phi(h^{1/2}A_h^{-1}x)}{h}~\text{where}~x\in\Omega_h= h^{-1/2}A_h \Omega.$$
By Lemma \ref{sep-lem}, $\phi_h$ and  $S_{\phi_h}(0, 1)$ satisfy the hypotheses of the Localization Theorem \ref{main_loc}
at all points $x_0\in\p\Omega_h\cap B(0, c)$. 
Note that
$$h^{-1/2}A_h S_\phi(0, t) = S_{\phi_h}(0, \frac{1}{\theta})$$
Therefore, from (\ref{small-sec}), we have 
\begin{equation}
\label{up-dist}
dist (Z, \p\Omega_h)\leq C~\text{for all}~ Z\in S_\phi(0,\frac{1}{\theta}).
\end{equation}
Because $h\leq c$, we have 
$Ch^{1/2}|\text{log }h|^2\leq 1/2.
$ Thus, for all $y\in S_\phi(0, t)$, we have $y_h=h^{-1/2}A_h  y\in \Omega_h\cap B_{k^{-1}}^+ $
and by (\ref{up-dist}) and (\ref{dist_h}), 
$$dist(y,\p\Omega) \leq (1+ Ch^{1/2}|\text{log }h|^2) h^{1/2} dist(y_h,\p\Omega_h) \leq Ch^{1/2}\leq Ct^{1/2}.$$
Hence (i) is proved.

For the proof of (ii),
we use (\ref{small-sec}) and the fact that the curvature of $\p\Omega_h\cap B_c$ is bounded by $Ch^{1/2}$ from Lemma \ref{sep-lem} (iii)
to find $Y\in\p S_{\phi_h}(0, \frac{1}{\theta})\cap\Omega_h$
such that 
$dist(Y, \p\Omega_h)\geq c'$ with $c'$ universally small.
Then (ii) follows easily because for $y= h^{1/2} A_h^{-1}Y\in \p S_\phi(0, t)\cap\Omega$, we have from (\ref{dist_h})
 $$dist(y,\p\Omega) \geq (1-Ch^{1/2}|\log h|^2) h^{1/2} dist(Y,\p\Omega_h) \geq C^{-1}t^{1/2}.$$
\end{proof}

\section{Proof of the boundary Harnack inequality}
\label{H_sec}
This section is devoted to the proofs of Theorems \ref{carl} and \ref{compa}. 

Throughout this section, we assume that 
the convex domain $\Omega$ and the convex function $\phi$ satisfy \eqref{global-tang-int}--\eqref{global-sep}. We will use the following universal constants in the previous sections:
\begin{center} \it
$\theta_{\ast}>1$ in Theorem \ref{engulfing2}, 
$c_0>0$ and $p_1=\bar\mu^{-1}>1$ in Theorem \ref{pst},\\
$\bar\mu$ in Proposition \ref{tan_sec} and Lemma \ref{size_sec}.
\end{center}
\subsection{Boundary properties of solutions to the linearized Monge-Amp\`ere equations}
In this section, we state several ingredients used in the proof of Theorem \ref{carl}. They are concerned with boundary properties 
of solutions to the linearized Monge-Amp\`ere equations.

The first ingredient in the proof of Theorem \ref{carl} states that the maximum value of a positive solution to
the linearized Monge-Amp\`ere equations on a boundary section of $\phi$
increases by a factor of $2$ when we pass to a concentric boundary section of universally larger height.
\begin{lemma} \label{max-2h}
Assume that $x_0\in \p\Omega$. Suppose that $u\geq 0$ is a continuous solution of $L_{\phi} u=0$ in $\Omega\cap S_\phi(x_0, c)$ and that $u=0$ on $\p\Omega \cap S_\phi(x_0, c)$.
 Then, there is a universal constant $H>1$ such that for all $0<h\leq c/2$, we have
 \begin{equation*}
  \max_{x\in \overline{S_\phi(x_0, h)}} u(x)\geq 2 \max_{x\in \overline{S_\phi(x_0, h/H)}} u(x).
 \end{equation*}
\end{lemma}

In the next ingredient in the proof of Theorem \ref{carl}, we show that for any interior point $x$ of the domain $\Omega$, we can find another interior point $y$ whose distance to the boundary 
is larger than that of $x$ and the values of any
nonnegative solution to the linearized Monge-Amp\`ere equation $L_\phi u=0$ are universally comparable at $x$ and $y$.

\begin{lemma} 
\label{small_dist} 
Let $\phi$ and $\Omega$ satisfy the hypotheses of the Localization Theorem \ref{main_loc} at the 
origin.
There exist universal constants $\hat K>1$ and $M_1>1$ with the following property. 
Suppose that $u\geq 0$ is a continuous solution of $L_{\phi} u=0$ in $\Omega\cap S_\phi(0, c)$ and that $u=0$ on $\p\Omega \cap S_\phi(0, c)$.
Assume that for some $x\in\Omega$ we have $\p\Omega \cap \p S_\phi(x,\bar h(x))=\{0\}$ and $dist^2(x,\p\Omega)/c\leq r$ where $r$ is universally small. 
Then, we can find $y\in S_\phi(0, \hat K \bar h(x))$ 
such that $M_1^{-1} u(x)\leq u(y)\leq  M_1 u(x),$
and $\bar h(y)\geq 2 \bar h(x).$
\end{lemma}

The final ingredient in the proof of Theorem \ref{carl} quantifies how close a point $x$ is to the boundary when a positive solution $u$ to the linearized Monge-Amp\`ere equation $L_\phi u=0$ is large 
at $x$ given a bound of $u$ at another point in the interior.
\begin{lemma}
\label{dist_decay} 
Let $\phi$ and $\Omega$ satisfy the hypotheses of the Localization Theorem \ref{main_loc} at the 
origin. Let $\hat K$ and $M_1$ be as in Lemma \ref{small_dist}.
Suppose that $u\geq 0$ is a continuous solution of $L_{\phi} u=0$ in $\Omega\cap S_\phi(0, c)$ and that $u=0$ on $\p\Omega \cap S_\phi(0, c)$ 
where $0\in\p\Omega$. Let $t/c\leq c_2$ be universally small.
Let $P$ be the {\it apogee} of $S_\phi(0, c_3 t)$ for some universal $c_3$ small with $\theta_{\ast}^2 (1+\hat K)\bar h(P)< c_0 t (1-2^{-1/p_1})^{p_1}$
(This $P$ exists by Lemma \ref{cen-lem}). Let $D_k= S_\phi(0, (2+ 2^{-k/p_1}) t)$.
If $x\in D_k$ and 
 $$u(x)\geq \hat C M_1^k u(P),$$
 for some universally large constant $\hat C$, then
$$\bar h(x)\leq 2^{-k} \bar h(P).$$
\end{lemma}
\begin{remark}
 By Proposition \ref{tan_sec}, $\bar h(x)$ and $dist^2 (x,\p\Omega)$ are comparable. However, $\bar h(x)$ is very well-behaved under affine transformations. That is why we use 
 $\bar h$ to measure
 the distance to the boundary in Lemmas \ref{small_dist}, \ref{dist_decay} and the proof of Theorem \ref{carl}.
\end{remark}

To prove the above ingredients and to implement the proof of Theorem \ref{carl}, we use geometric properties 
of boundary sections established in Section \ref{quasi_sec}. Moreover, we will use some geometric properties of the quasi distance 
$\delta_\phi$ generated by $\phi$. For $x\in\overline{\Omega}$, let us introduce the function $\delta_{\phi}(x,\cdot):\overline{\Omega}\rightarrow [0,\infty)$ defined by
\begin{equation}\delta_\phi(x, y) =\phi(y)-\phi(x)-\nabla\phi(x)(y-x).
 \label{del_dist}
\end{equation}
Then, we have the following quasi-metric inequality for $\delta_\phi$.
\begin{lemma}\label{quasi}  Assume that $\Omega$ and $\phi$ satisfy \eqref{global-tang-int}--\eqref{global-sep}. Then, for all $ x, y, z\in\overline{\Omega}$, we have
 $$\delta_\phi(x, y) \leq \theta_{\ast}^2 (\delta_\phi(x, z) + \delta_\phi(z, y)).$$
\end{lemma}
\begin{proof}[Proof of Lemma \ref{quasi}]
 Let
 $r= \delta_\phi(x, z)$, and $s= \delta_\phi(z, y).$
Then for any $\varepsilon>0$, we have $y\in S_\phi(z, s+ \varepsilon)\subset S_\phi(z, \theta_{\ast}(r+ s +\varepsilon))$
and $z\in S_\phi(x, r +\varepsilon)$. Thus, by Theorem \ref{engulfing2}, we have
$$S_\phi(x, r+\varepsilon)\subset S_\phi(z, \theta_{\ast}(r+\varepsilon)) \subset S_\phi(z, \theta_{\ast}(r+s+ \varepsilon)).$$
Again, by Theorem \ref{engulfing2}, we have
$$S_\phi(z, \theta_{\ast}(r+s+ \varepsilon)) \subset S_\phi(x, \theta^2_{\ast}(r+s+ \varepsilon)).$$
It follows that $y\in  S_\phi(x, \theta^2_{\ast}(r+s+ \varepsilon))$. Hence
$\delta_\phi(x, y)\leq \theta^2_{\ast}(r+s+ \varepsilon)$
and the conclusion of the lemma follows by letting $\varepsilon\rightarrow 0.$
 \end{proof}

With the help of Theorem \ref{pst}, we can now follow Maldonado's proof of  \cite[Lemma 4]{M} to show that $\delta_{\phi}(x,\cdot)$ satisfies a uniform H\"older property 
as stated in the following lemma. Recalling that the constant $p_1$ in Theorem \ref{pst} satisfies $p_1=\bar\mu^{-1}$ where $\bar\mu$ is the constant in Lemma \ref{size_sec}.
 \begin{lemma}\label{bdr-H-d} Assume that $\Omega$ and $\phi$ satisfy \eqref{global-tang-int}--\eqref{global-sep}.
 Let $\bar\mu$ be as in Proposition \ref{tan_sec} and Lemma \ref{size_sec}.
 Then, there exists a universal constant $C_2\geq 1$ such that for all $x, y, z\in\overline{\Omega}$, we have
 $$|\delta_\phi(x, z)-\delta_\phi(x, y)|\leq C_2 [\delta_\phi(y, z)]^{\bar\mu} \left(\delta_\phi(x, z) + \delta_\phi(x, y)\right)^{1-\bar\mu}.$$
\end{lemma}

We are now ready to prove the first main result of the paper.
\begin{proof}[Proof of Theorem \ref{carl}] 
We can assume that $x_0=0\in\p\Omega$ and $\phi$ and $\Omega$ satisfy the hypotheses of the Localization Theorem \ref{main_loc} at 
0. Let $\hat K$ and $c_2$ be as in Lemma \ref{small_dist}; $\hat C, c_3$ and $M_1$ as in Lemma \ref{dist_decay}.
Let $\bar t= c_2 t$.
Then the {\it apogee} P of $S_\phi(0, c_3\bar t)$ satisfies $\theta_{\ast}^2 (1+\hat K)\bar h(P)< c_0\bar t(1-2^{-1/p_1})^{p_1}$. 
Without loss of generality, we assume that $u>0$ in $\Omega$.
We prove that
\begin{equation}
 \label{Key_Ht}
 u(x) < \hat C M_1^{k} u(P)~\text{for all}~ x\in S_\phi(0, c_3\bar t/2)
\end{equation}
for some universally large constant $k$ to be determined later on. Then by using Lemma \ref{chain_lem} and
the interior Harnack inequality for the linearized Monge-Amp\`ere equation in Theorem \ref{Holder_thm}, we easily obtain the stated estimate of the theorem.

We will use a contradiction argument, following the main lines of the proof of the Carleson estimate in \cite{CFMS}. Suppose that (\ref{Key_Ht}) is false. Then we can find $x_1\in S_\phi(0, c_3 \bar t/2)\subset S_\phi(0, (2+ 2^{-k/p_1})\bar t)$ satisfying $u(x_1)\geq \hat CM_1^k u(P).$ 
Invoking Lemma \ref{dist_decay}, we find
\begin{equation}\bar h(x_1)\leq 2^{-k} \bar h(P).
 \label{hx1bound}
\end{equation}
We fix a universally large integer $s$ so that 
$2^s\geq M_1.$ Let $z=\p\Omega\cap\p S_\phi(x_1,\bar h(x_1))$. Then, by Theorem \ref{engulfing2},
$$S_\phi(x_1,\bar h(x_1))\subset S_\phi(z,\theta_{\ast}\bar h(x_1))\subset S_\phi(z, H^s\theta_{\ast}\bar h(x_1)),$$
where $H$ is as in Lemma \ref{max-2h}. Reducing $c_3$ if necessary, we have for $k$ universally large
\begin{equation}
 \label{szH}
 S_\phi(z, 2H^s\theta_{\ast}\bar h(x_1))\subset S_\phi(0, \bar t/2).
\end{equation}
Indeed, if $y\in S_\phi(z, 2H^s\theta_{\ast}\bar h(x_1))$ then $\delta_\phi(z, y)\leq 2H^s\theta_{\ast}\bar h(x_1)$. From $\delta_\phi(x_1, z)\leq \bar h(x_1)$,
$\bar h(P)\leq \bar t$, and Lemma \ref{quasi},
we have in view of (\ref{hx1bound}),
$$\delta_\phi(x_1, y)\leq \theta_\ast^2 (\delta_\phi(x_1, z) + \delta_\phi(z, y))\leq \theta_{\ast}^2 (1 + 2H^s\theta_{\ast})\bar h(x_1)\leq 
\theta_{\ast}^2 (1 + 2H^s\theta_{\ast}) 2^{-k} \bar t.$$
Thus, from $\delta_\phi(0, x_1)\leq c_3 \bar t/2$ and Lemma \ref{quasi}, we obtain
$$\delta_\phi(0, y)\leq  \theta_\ast^2 (\delta_\phi(0, x_1) + \delta_\phi(x_1, y))\leq \theta_\ast^2 [c_3 \bar t/2 + 
\theta_{\ast}^2 (1 + 2H^s\theta_{\ast}) 2^{-k} \bar t] < \bar t/2$$
if $c_3$ is universally small and $k$ universally large.

With (\ref{szH}), applying Lemma \ref{max-2h} to $S_\phi(z, H^m\theta_{\ast}\bar h(x_1))$ successively for $m=0,\cdots,s-1$, we can find $x_2\in 
\overline{S_\phi(z, H^s\theta_{\ast}\bar h(x_1))}$ such that 
\begin{equation}u(x_2)\geq 2^s u(x_1) \geq \hat C M_1^{k+1} u(P).
\label{ux2bound}
\end{equation}
From $\delta_\phi(x_1, z)\leq \bar h(x_1),~\delta_\phi(z, x_2)\leq H^s \theta_{\ast} \bar h(x_1)$ and Lemma \ref{quasi}, we have
$$\delta_\phi(x_1, x_2)\leq \theta_{\ast}^2(1+ \theta_{\ast} H^s) \bar h(x_1).$$
Let $$d_0=\delta_\phi(0, x_{1}), d_1= \delta_\phi(0, x_{2}).$$
Now, using $\bar h(P)\leq \bar t$, (\ref{hx1bound}) and the H\"older property of $\delta_\phi(0,\cdot)$ in Lemma \ref{bdr-H-d}, we find
\begin{eqnarray} d_1 &\leq& d_0 + C_2(\delta_\phi(x_1, x_{2}))^{\bar\mu} (d_0 + d_1)^{1-\bar\mu}  \leq d_0 + C_2 \left(\theta_{\ast}^2(1+ 
\theta_{\ast} H^s) 2^{-k}\bar h(P)\right)^{\bar\mu} (d_0 + d_1)^{1-\bar\mu}
\nonumber\\ &\leq & d_0 + \theta_0^{\bar\mu} (d_0 + d_1)^{1-\bar\mu}
\label{d1bound}
\end{eqnarray}
where $\{\theta_j\}_{j=0}^{\infty}$ is the sequence defined by
\begin{equation}\theta_j = C_2^{1/\bar\mu} \theta_{\ast}^2(1+ \theta_{\ast} H^s)2^{-(k+ j)}\bar t.  
 \label{thetaj}
\end{equation}
Clearly, if $k$ is universally large, we have $\theta_{j+1}<\theta_j<c_3 \bar t/2$ for all $j$ and furthermore
\begin{equation}\frac{2}{1- (\frac{\theta_0}{\bar t/2})^{\bar\mu}} \sum_{j=0}^{\infty}\left( \frac{\theta_j}{\bar t/2}\right)^{\bar\mu}<\log\frac{3}{2}.
 \label{kchoice}
\end{equation}
Given this $k$, we know from (\ref{d1bound}) together with $d_0<c_3 \bar t/2$ and the next Lemma \ref{ims-lem} that $\delta_\phi(0, x_{2}) < 3c_3\bar t/4$.  
Hence $x_2\in S_\phi(0, 3c_3 \bar t/4)$. Recalling (\ref{ux2bound}), we conclude
from Lemma \ref{dist_decay} that 
$$\bar h(x_2)\leq 2^{-(k+1)} \bar h(P).$$
With $k$ satisfying (\ref{kchoice}), we repeat the above process to obtain 
a sequence $\left\{x_j\right\}\subset S_\phi(0, 3c_3 \bar t/4)$ such that
\begin{equation}u(x_j)\geq \hat CM_1^{k+ j-1} u(P)
 \label{uxj}
\end{equation}
with
\begin{equation}\bar h(x_j)\leq 2^{-(k+j-1)}\bar h (P),~\text{and}~
\delta_\phi(x_j, x_{j+1})\leq \theta_{\ast}^2(1+ \theta_{\ast} H^s)\bar h(x_j)
\label{hxj}
\end{equation}
provided that $x_j\in S_\phi(0, 3c_3 \bar t/4)$. But this follows from the choice of $k$. Indeed, recalling $\theta_j$ in (\ref{thetaj}) and 
the H\"older property of $\delta_\phi(0,\cdot)$ in Lemma \ref{bdr-H-d}, we obtain 
for $d_{j}=\delta_\phi(0, x_{j+1})$
the estimate
$$d_{j+1}- d_{j} \leq C_2(\delta_\phi(x_{j+1}, x_{j+2}))^{\bar\mu}  (d_j + d_{j+1})^{1-\bar\mu} \leq \theta_j^{\bar\mu}(d_j + d_{j+1})^{1-\bar\mu}. $$
Hence, from $d_0\leq c_3 \bar t/2$, and (\ref{kchoice}),  we find from Lemma \ref{ims-lem} that
$$d_{j+1}\leq \frac{c_3 \bar t}{2} \text{exp}\left[\sum_{j=0}^{\infty} 
\frac{2}{1- (\frac{\theta_0}{\bar t/2})^{\bar\mu}}\left( \frac{\theta_j}{\bar t/2}\right)^{\bar\mu}\right]<\frac{3c_3}{4}\bar t.$$
We now let $j\rightarrow \infty$ in (\ref{uxj}) and (\ref{hxj}) to obtain $x_{\infty}\in \p\Omega\cap S_\phi(0, c_3\bar t)$ with $u(x_{\infty})=\infty$. This is a contradiction
to $u(x_{\infty})=0$. Hence (\ref{Key_Ht}) holds and the proof of our Theorem is
complete.
\end{proof}
In the proof of Theorem \ref{carl}, we used the following lemma, taken from \cite[Lemma 1]{IMS}.
\begin{lemma} (\cite[Lemma 1]{IMS})
\label{ims-lem}
 Let $R>0$, $\alpha\in (0, 1)$, $\{d_n\}_{n=0}^{\infty}\subset (0,\infty)$, $\{\theta_n\}_{n=0}^{\infty}\subset (0, R)$ such that $d_0\leq R,$
 $$d_{n+1}\leq d_n +\theta_n^{\alpha}(d_n + d_{n+1})^{1-\alpha},~
 \text{and~}
 \theta_{n+1}<\theta_n<R~\text{for all n}.$$
 Then 
 $$d_{n+1}\leq R~ \exp\left[ \frac{2}{1- (\theta_0/R)^{\alpha}} \sum_{j=0}^{n}\left( \frac{\theta_j}{R}\right)^{\alpha}\right].$$
\end{lemma}
\begin{proof}[Proof of Theorem \ref{compa}]
We just sketch the proof. By multiplying $u$ and $v$ by suitable constants, we assume that $u(P)= v(P)=1$.
Without loss of generality, we assume that $x_0=0$, $\phi$ and $\Omega$ satisfy the hypotheses of the Localization Theorem \ref{main_loc} at the origin. 
By Theorem \ref{carl},
\begin{equation}u(x)\leq M_0~\text{for all} ~x\in S_\phi(0, 3t/4).
 \label{uMeq}
\end{equation}
By \cite[Proposition 6.1]{LS}, we have
\begin{equation}u(x)\leq CM_0 t^{-1/2}dist (x,\p\Omega)~ \text{for all} ~x\in S_\phi(0, t/2).
\label{u-dist}
\end{equation}
Using Lemma \ref{chain_lem} and the interior Harnack inequality in Theorem \ref{Holder_thm}, we can find a universal constant $c_1$ such that
$$v(x)\geq c_1
~\text{for all} ~x\in S_\phi(0, 3t/4)~ \text{with}~ dist(x,\p\Omega)\geq c_1 t^{1/2}.$$
It follows from \cite[Proposition 5.7]{LS} that
\begin{equation}v(x)\geq c_2 t^{-1/2}dist (x,\p\Omega)~ \text{for all} ~x\in S_\phi(0, t/2).
\label{v-dist}
\end{equation}
The theorem follows from (\ref{u-dist}) and (\ref{v-dist}).

For the sake of completeness, we indicate how to obtain (\ref{u-dist}) and (\ref{v-dist}). Let $h=t$. We use the Localization Theorem \ref{main_loc} at the origin and consider the rescaled functions
$$\phi_h(x)= \frac{\phi(h^{1/2}A_h^{-1}x)}{h}, u_h(x)= u(h^{1/2}A_h^{-1}x)~\text{where}~x\in\Omega_h= h^{-1/2}A_h \Omega.$$
Then $u_h$ is a positive solution of
$$L_{\phi_h} u_h=0~ \text{in}~ \Omega_h\cap S_{\phi_h}(0, 1)~ \text{and}~ u_h=0~ \text{on}~ \p\Omega_h\cap S_{\phi_h}(0, 1).$$
By \cite[Proposition 6.1]{LS}, we know that  
$$ u_h(x) \leq C \left(\max_{S_{\phi_h}(0, 3/4)} u_h\right)  dist (x,\p\Omega_h)~\text{for all}~x \in S_{\phi_h}(0, 1/2).$$
Using (\ref{uMeq}) and (\ref{dist_h}), we obtain (\ref{u-dist}). The proof of (\ref{v-dist}), using rescaling, is similar.
\end{proof}
\subsection{Proofs of boundary properties of solutions to the linearized Monge-Amp\`ere}
\begin{proof}[Proof of Lemma \ref{max-2h}] Without loss of generality, we assume that $x_0=0$
and $\phi$ and $\Omega$ satisfy the hypotheses of the Localization Theorem \ref{main_loc} at 
0. We use the Localization Theorem \ref{main_loc} at the origin and consider the rescaled functions
$$\phi_h(x)= \frac{\phi(h^{1/2}A_h^{-1}x)}{h}, u_h(x)= u(h^{1/2}A_h^{-1}x)~\text{where}~x\in\Omega_h= h^{-1/2}A_h \Omega.$$
Then $u_h$ is a positive solution of
$$L_{\phi_h} u_h=0~ \text{in}~ \Omega_h\cap S_{\phi_h}(0, 2)~ \text{and}~ u_h=0~ \text{on}~ \p\Omega_h\cap S_{\phi_h}(0, 2).$$
We need to show that for $H$ universally large
\begin{equation*}
  \max_{x\in \overline{S_{\phi_h}(0, 1)}} u_h(x)\geq 2 \max_{x\in \overline{S_{\phi_h}(0, 1/H)}} u_h(x).
 \end{equation*}
Indeed, by Lemma \ref{sep-lem}, $\phi_h$ and  $S_{\phi_h}(0, 1)$ satisfy the hypotheses of the Localization Theorem \ref{main_loc}
at all points $x_0\in\p\Omega_h\cap B(0, c)$. Hence, if $1/H\leq c$ then from (\ref{small-sec}), we have
$$ \max_{x\in \overline{S_{\phi_h}(0, 1/H)}} dist (x,\p\Omega_h)\leq\max_{x\in \overline{S_{\phi_h}(0, 1/H)}} |x|\leq C_2 (\frac{1}{H})^{1/2} \log (\frac{1}{H}).$$
On the other hand, by \cite[Proposition 6.1]{LS}, we know that  
$$\max_{x\in \overline{S_{\phi_h}(0, 1/H)}} u_h(x) \leq C \max_{x\in \overline{S_{\phi_h}(0, 1)}} u_h(x) \max_{x\in \overline{S_{\phi_h}(0, 1/H)}} dist (x,\p\Omega_h)$$
for some $C= C(n,\lambda,\Lambda, \rho)$.
Therefore, for $H$ universally large. 
$$\max_{x\in \overline{S_{\phi_h}(0, 1/H)}} u_h(x) \leq C C_2 (\frac{1}{H})^{1/2} \log (\frac{1}{H})\max_{x\in \overline{S_{\phi_h}(0, 1)}} u_h(x)\leq 
\frac{1}{2}\max_{x\in \overline{S_{\phi_h}(0, 1)}} u_h(x).$$
\end{proof}
\begin{proof}[Proof of Lemma \ref{small_dist}]
Let $\{x_0=x, x_1,\cdots, x_m\}\subset S_\phi(0, \bar K dist^2(x,\p\Omega))$ and $\bar K$ be as in Lemma \ref{chain_lem} (i). 
We deduce from Proposition \ref{tan_sec} that $k_0^{2} \bar h(x) \leq dist^2(x,\p\Omega)\leq k_0^{-2} \bar h(x)$. Hence, for $\hat K:= \bar K k_0^{-2}$, 
$$x_m\in S_\phi(0, \hat K \bar h(x)),~
\text{and}~
 \bar h(x_m)\geq k_0^2 dist^2(x_m,\p\Omega)\geq 4k_0^{-2} dist^2(x,\p\Omega)\geq 4 \bar h(x).$$
 If $dist^2(x,\p\Omega)/c\leq r$ where $r$ is universally small then $$S_\phi(0, \hat K \bar h(x))\subset S_\phi(0, c/2)~\text{and}~ S_\phi(x_i,\bar h(x_i))\subset S_\phi(0, c/2)~\text{for all }i.$$
Since $u\geq 0$ satisfies $L_\phi u=0$ in $S_\phi(x_{i+1}, \frac{1}{2}\bar h(x_{i+1}))\subset\subset S_\phi(0, c/2)\cap \Omega$ and
$x_i\in S_\phi (x_{i+1},\tau \bar h(x_{i+1}))$, we obtain from the Harnack inequality in 
Theorem \ref{Holder_thm}
\begin{equation}
 \label{hitau}
 C_1^{-1} u (x_i)\leq u(x_{i+1})\leq C_1(n,\lambda,\Lambda) u (x_i)~\text{for each~}i=0, 1,\cdots, m-1.
\end{equation}
We now take $y=x_m$. Then $\bar h(y)\geq 4\bar h(x)$. From the chain condition $x_i\in S_\phi(x_{i+1},\tau\bar h(x_{i+1}))$ for all $i=0, 1,\cdots, m-1$ with $m$ universally bounded
 and (\ref{hitau}), we obtain for some universal $M_1$: $$M_1^{-1} u(x)\leq u(y)\leq  M_1 u(x).$$
\end{proof}
\begin{proof}[Proof of Lemma \ref{dist_decay}] By using Lemma \ref{chain_lem} and the interior Harnack inequality for the linearized Monge-Amp\`ere equation in Theorem \ref{Holder_thm}, we can find a universally large constant $\hat C$ such that for all $x\in S_\phi(0, 3t)$ with $\bar h(x)\geq 2^{-1}\bar h(P)$, we have $u(x)\leq (\hat C/2) u(P)$.
With this choice of $\hat C$, we can now prove statement of the lemma
by induction.
Clearly, the statement is true for $k=1$. Suppose it is true for $k\geq 1$. We prove it for $k+1$. Let $x\in D_{k+1}$ be such that 
$$u(x)\geq \hat CM_1^{k+1} u(P).$$
Let $z=\p\Omega\cap\p S_\phi(x,\bar h(x))$. If $t/c\leq c_2$ is small then we are in the setting of Lemma \ref{small_dist} with $z$ replacing $0$.
Indeed, we first note from $x\in D_{k+1}\subset S_\phi(0, 3t)$ and Lemma \ref{cen-lem} that 
\begin{equation}
 dist^2 (x,\p\Omega)\leq Ct.
 \label{Az}
\end{equation}
By Proposition \ref{tan_sec}, $$\delta_\phi(x, z)=\bar h(x)\leq k_0^{-2} dist^2(x,\p\Omega)\leq Ck_0^{-2}t.$$
Using Lemma \ref{quasi}, and $\delta_\phi(0, x)\leq 3t$, we find
$$\delta_\phi(0, z)\leq \theta_{\ast}^2[\delta_\phi(0, x) + \delta_\phi(x, z)]\leq C't$$
for some universal $C'$. Now, let $M=C/r$ where $C$ is as in (\ref{Az}) and $r$ is as in Lemma \ref{small_dist}. If $y\in S_\phi(z, Mt)$ then from Lemma \ref{quasi}, we have
$$\delta_\phi(0, y)\leq \theta_{\ast}^2[\delta_\phi(0, z) + \delta_\phi(z, y)]\leq \theta_{\ast}^2[C't + Mt]\leq c/2$$
provided that $t/c\leq c_2$ is universally small. It follows that if $c_2$ is small then
$
 S_\phi(z, Mt)\subset S_\phi(0, c/2).$
This combined with (\ref{Az}) and the choice of $M$ shows that the hypotheses of Lemma \ref{small_dist} are satisfied. 

By Lemma \ref{small_dist}, we can find $y\in S_\phi(z, \hat K\bar h(x))$ such that 
$$u(y)\geq M_1^{-1} u(x)\geq \hat CM_1^k u(P),~\text{and}~ \bar h(y)\geq 2 \bar h(x).$$
The lemma follows if we can show $y\in D_k$. Indeed, 
by the induction hypothesis for $k$, we have
$\bar h(y)\leq 2^{-k} \bar h(P)$
and hence
$$\bar h(x)\leq \frac{1}{2}\bar h (y)\leq 2^{-(k+1)} \bar h(P).$$
It remains to show that $y\in D_{k}$.
From $\delta_\phi(x, z)= \bar h(x),~\delta_\phi(z, y)\leq \hat K \bar h(x)$ and Lemma \ref{quasi}, we have
$$\delta_\phi(x, y)\leq \theta_\ast^2(1+ \hat K) \bar h(x).$$
Let us denote for simplicity $r_k = 2 + 2^{-k/p_1}$. Then $D_k= S_\phi(0, r_k t)$. Recalling $x\in D_{k+1}= S_\phi(0, r_{k+1}t)$, we have from Theorem \ref{pst} (i), 
$$S_\phi(x, c_0 (r_k- r_{k+1})^{p_1}t) \subset S_\phi(0, r_k t)= D_k.$$
By our choice of $P$, we have $\theta_\ast^2(1+ \hat K)\bar h(P)< c_0t(1-2^{-1/p_1})^{p_1}$ and hence
$$c_0(r_k- r_{k+1})^{p_1}t = c_0t2^{-k} (1-2^{-1/p_1})^{p_1}>\theta_\ast^2(1+ \hat K) 2^{-k}\bar h(P)\geq \theta_\ast^2(1+ \hat K)\bar h(x)\geq \delta_\phi(x, y).$$
Therefore,
$y\in S_\phi(x, c_0 (r_k- r_{k+1})^{p_1}t) \subset S_\phi(0, r_k t)= D_k.$
\end{proof}

\section{Bounding the Green's function for the linearized Monge-Amp\`ere operator }
\label{G_sec}
In this section, we prove Theorem \ref{Gthm} and Corollary \ref{lpbound} using the Harnack inequalities in Theorems \ref{carl}, \ref{Holder_thm} and 
properties of the Green's function for  the linearized Monge-Amp\`ere operator  $L_\phi$ as stated in the next subsection.

Throughout, we assume $\Omega$ and $\phi$ satisfy \eqref{global-tang-int}--\eqref{global-sep}.
Recall that $g_{V}(x, y)$ denotes the Green's function of $L_\phi$ in $V$ with pole $y\in V\cap\Omega$ where $V\subset\overline{\Omega}$, that is
$g_V(\cdot, y)$ is a positive solution of 
$$L_\phi g_V(\cdot, y)=\delta_{y}~\text{in} ~V\cap\Omega,~\text{and}~
 g_V(\cdot, y) =0~\text{on}~\p V.$$
\subsection{Properties of the Green's function} First, from the Aleksandrov-Bakelman-Pucci maximum principle, 
we observe the following uniform $L^{\frac{n}{n-1}}$  and $L^1$ bounds for the Green's function:

\begin{lemma} \label{abplem}
 If $V\subset\Omega$ is a convex domain then for all $x_0\in V$, we have the uniform $L^{\frac{n}{n-1}}$ bound:
 $$\left(\int_V \left[g_{V}(x, x_0)\right]^{\frac{n}{n-1}} dx\right)^{\frac{n-1}{n}}\leq C(n,\lambda) |V|^{1/n}.$$
As a consequence, we have the uniform 
$L^{1}$ bound:
 $$\int_V g_{V}(x, x_0) dx\leq C(n,\lambda) |V|^{2/n}.$$
 \end{lemma}
Given an $L^1$ bound for positive solution to the linearized Monge-Amp\`ere equation, we can use the interior Harnack inequality in Theorem \ref{Holder_thm} to 
obtain a pointwise upper bound in compactly supported subsets via the following general estimate.
\begin{lemma} \label{genup}Suppose $S_\phi(x_0, t)\subset\subset\Omega$.
 If $\sigma\geq 0$ satisfies $L_\phi\sigma=0$ in $S_\phi(x_0, t)\backslash \{x_0\}$ and  
 $$\int_{S_\phi(x_0, t)}\sigma(x) dx\leq A$$
 then
 $$\sigma(x)\leq C(n,\lambda,\Lambda)At^{-\frac{n}{2}}~\forall x\in \p S_{\phi}(x_0, t/2).$$
\end{lemma}

Using the Harnack inequalities for the  linearized Monge-Amp\`ere equations in Theorems \ref{carl}, \ref{Holder_thm}, we can give
in the next lemma a sharp upper bound for the Green's function $g_{S_\phi(x_0, 2t)}(z, x_0)$ in the concentric section with half-height.
It is a global version of \cite[Lemma 3.1]{L}.
\begin{lemma}\label{doubleh}
 If $x_0\in\Omega$ and $t\leq c(n,\lambda,\Lambda, \rho)$ then
 \begin{equation}
  \max_{z\in \p S_\phi(x_0, t)} g_{S_\phi(x_0, 2t)}(z, x_0)\leq C(n,\lambda,\Lambda, \rho)t^{-\frac{n-2}{2}}.
  \label{dh}
 \end{equation}
\end{lemma}
The following lemma estimates the growth of the Green's function near the pole. It is a variant at the boundary of \cite[Lemma 3.2]{L}.
\begin{lemma}\label{todh}
If $x_0\in \Omega$, then 
\begin{equation}\max_{x\in\p S_\phi(x_0, t)} g_\Omega(x, x_0)\leq C(n,\lambda,\Lambda, \rho)t^{-\frac{n-2}{2}} + \max_{z\in \p S_\phi(x_0, 2t)} g_\Omega(z, x_0).
\label{iter_max}
\end{equation}
\end{lemma}

With the above properties of the Green's function, we are now ready to prove Theorem \ref{Gthm}.
\begin{proof} [Proof of Theorem \ref{Gthm}]
Let $m$ be a nonnegative integer such that $c/2< 2^m t_0\leq c$.
 Iterating Lemma \ref{todh}, we find
\begin{equation}\max_{x\in\p S_\phi(x_0, t_0)} g_{\Omega}(x, x_0)
\leq \max_{x\in\p S_\phi(x_0, 2^mt_0)} g_{\Omega}(x, x_0) + Ct_0^{-\frac{n-2}{2}}\sum_{l=1}^m (2^l)^{-\frac{n-2}{2}}.
\label{iterG}
\end{equation}
The desired upper bound for $\max_{x\in\p S_\phi(x_0, t_0)} g_{\Omega}(x, x_0)$ follows from the following:\\
{\bf Claim.} $\sigma:=g_{\Omega}(\cdot, x_0)$ is universally bounded by $C(n,\lambda,\Lambda, \rho)$ in $\Omega\backslash S_\phi(x_0, c/2)$.\\
Indeed, if $n\geq 3$, then $$\sum_{l=1}^m (2^l)^{-\frac{n-2}{2}} \leq C(n),$$
and hence, the Claim and (\ref{iterG}) give
$$\max_{x\in\p S_\phi(x_0, t_0)} g_{\Omega}(x, x_0) \leq C(n,\lambda,\Lambda, \rho) +  CC(n) t_0^{-\frac{n-2}{2}} \leq C(n,\lambda,\Lambda,\rho) t_0^{-\frac{n-2}{2}}.$$
If $n=2$, then since $m\leq \log_2 (c/t_0)$, the Claim and (\ref{iterG}) give
$$\max_{x\in\p S_\phi(x_0, t_0)} g_{\Omega}(x, x_0) \leq C(n,\lambda,\Lambda, \rho) +  C \log_2 (c/t_0) \leq C(n,\lambda,\Lambda,\rho) t_0^{-\frac{n-2}{2}} |\text{log}~ t_0|.$$

It remains to prove the Claim. If $S_{\phi}(x_0, c/4)\subset\Omega$ then by 
H\"older inequality and Lemma \ref{abplem}, 
 \begin{eqnarray*}
  \int_{S_\phi(x_0, c/4)} \sigma \leq \|\sigma\|_{L^{\frac{n}{n-1}}(\Omega)} |S_\phi(x_0, c/4)|^{\frac{1}{n}}\leq C(n,\lambda)
  |\Omega|^{\frac{1}{n}} |S_\phi(x_0, c/4)|^{\frac{1}{n}} \leq C.
 \end{eqnarray*}
We used the volume estimate in the last inequality. 
By Lemma \ref{genup}, 
$$\max_{x\in\p S_\phi(x_0, c/8)} \sigma(x) \leq C.$$
The claim follows from the maximum principle.

Suppose now $S_{\phi}(x_0, c/4)\not\subset\Omega$. We consider the ring $A=S_\phi(x_0, c)\backslash
S_\phi(x_0, c/4)$ and focus on the values of $\sigma$ at points $z\in \p S_{\phi}(x_0, c/2)\cap\Omega$. 
From the geometric properties of sections in Theorem \ref{pst}, we can find a section 
\begin{equation}S_\phi(z, c_2)\subset\subset A\cap\Omega
 \label{szc2}
\end{equation}
for $z\in\Omega$ and some $c_2$ universally small. For reader's convenience, we indicate how to construct this section. Since $S_{\phi}(x_0, c/4)\not\subset\Omega$, we can 
find $y\in S_\phi(x_0, c/2)\cap\p\Omega$. Then Theorem \ref{pst} gives $S_\phi(y, 2c_3)\subset S_\phi(x_0, 2c/3)\backslash S_\phi(x_0, c/3)$ for some universal constant $c_3>0$.
Let $z$ be the {\it apogee} of $S_\phi(y, 2c_3)$ (see Lemma \ref{cen-lem}), that is $z\in \p S_\phi(y, c_3)\cap\Omega$ with $dist(z,\p\Omega)\geq c_4$ for some universal $c_4>0$. We now
use Proposition \ref{tan_sec} to obtain a section $S_\phi(z, c_2)\subset\subset A\cap\Omega$ as asserted in (\ref{szc2}).

The uniform $L^1$ bound for $\sigma$ from Lemma \ref{abplem} gives
\begin{equation}\label{l1sig} \int_{S_\phi(z, c_2)}\sigma(x) dx \leq \int_{\Omega}\sigma(x) dx \leq C(n,\lambda) |\Omega|^{2/n},
 \end{equation}
Note that $L_\phi \sigma=0$ in $S_\phi(z, c_2)$. By the interior Harnack inequality in Theorem \ref{Holder_thm} and the lower volume bound in Lemma \ref{vol_lem}, we obtain
from (\ref{l1sig})
\begin{equation}
 \sigma (z) \leq C(n,\lambda,\Lambda,\rho).
 \label{szup}
\end{equation}
On the other hand, using the chain property of sections in Lemma \ref{chain_lem} and the boundary Harnack inequality in Theorem \ref{carl}, we find that $$\sigma(y)\leq
C(n,\lambda,\Lambda,\rho)\sigma(z)~ \text{for all}~ y\in \p S_\phi(x_0, c/2)\cap\Omega.$$ 
Combining this with (\ref{szup}), we find that
$$\sigma(y)\leq
C(n,\lambda,\Lambda,\rho)~ \text{for all}~ y\in \p S_\phi(x_0, c/2)\cap\Omega.$$ 
The claim now also follows from the maximum principle.
\end{proof}
We now give the proof of Corollary \ref{lpbound}.
\begin{proof}[Proof of Corollary \ref{lpbound}]
We give the proof for $n\geq 3$; the case $n=2$ is similar. Let $x_0\in\Omega$.
Let $C, c$ be the constants in Theorem \ref{Gthm}. Denote by $C_{\ast}= C/c^{\frac{n-2}{2}}$. For $s\geq C_\ast$, we have $(\frac{C}{s})^{\frac{2}{n-2}}\leq c$. Hence, from Theorem \ref{Gthm} 
$$\{x\in\Omega: g_{\Omega}(x, x_0)>s\}\subset S_{\phi}(x_0, (C/s)^{\frac{2}{n-2}}).$$
Then, by the volume estimate in Lemma \ref{global-gradient} (iii), we have
$$|\{x\in\Omega: g_{\Omega}(x, x_0)>s\}|\leq |S_{\phi}(x_0, (C/s)^{\frac{2}{n-2}})|\leq C_1 |(C/s)^{\frac{2}{n-2}}|^{\frac{n}{2}}
= C_3 s^{-\frac{n}{n-2}}.$$
It follows from the layer cake representation that
\begin{eqnarray*}
 \int_{\Omega}g_{\Omega}^p(x, x_0) dx &=&\int_{0}^{\infty} ps^{p-1} |\{x\in\Omega: g_{\Omega}(x, x_0)>s\}| ds
 \\&=& \int_{0}^{C_\ast} ps^{p-1} |\{x\in\Omega: g_{\Omega}(x, x_0)>s\}| ds + \int_{C_\ast}^{\infty} p s^{p-1} |\{x\in\Omega: g_{\Omega}(x, x_0)>s\}| ds\\
 &\leq & C_\ast^p|\Omega| + \int_{C_\ast}^{\infty} ps^{p-1} C_3 s^{-\frac{n}{n-2}} ds
 =C_\ast^p|\Omega| + \frac{pC_3}{\frac{n}{n-2}-p} C_\ast ^{p- \frac{n}{n-2}}\leq C(n,\lambda,\Lambda,\rho, p).
\end{eqnarray*}
We now prove last inequality stated in the Corollary. Fix $x\in  S_\phi(x_0, t)\cap\Omega$ where $t\leq c_1:=\frac{c}{2\theta_{\ast}}$. Then, by the engulfing property of sections 
in Theorem \ref{engulfing2}, we have $ S_\phi(x_0, t)\subset  S_\phi(x, \theta_{\ast}t).$ Therefore, if $y\in S_\phi(x_0, t)$
then $y\in S_\phi(x, \theta_{\ast}t)\subset S_\phi(x, 2\theta_{\ast}t)$. Note that
\begin{equation}
\int_{S_\phi(x_0, t)} g_{S_\phi(x_0, t)}^p (y, x) dy\leq \int_{S_\phi(x, 2\theta_{\ast}t)} g_{S_\phi(x, 2\theta_{\ast}t)}^p (y, x) dy
 \label{upg}
\end{equation}
and it suffices to bound from above the term on the right hand side of (\ref{upg}). 

As in the above upper bound for  $\int_{\Omega}g_{\Omega}^p(x, x_0) dx $, if we replace 
$\Omega$ by $S_\phi(x, 2\theta_{\ast}t)$ where we recall $2\theta_{\ast}t\leq c$ and $C_{\ast}$ by $C_0= C/t^{\frac{n-2}{2}}$ then 
for  $s\geq C_0$, 
$$\{y\in S_\phi(x, 2\theta_{\ast}t): g_{S_\phi(x, 2\theta_{\ast}t)}(y, x)>s\}\subset \{y\in\Omega: g_{\Omega}(y, x)>s\}\subset S_{\phi}(x, (C/s)^{\frac{2}{n-2}}).$$
Consequently, we also obtain
\begin{equation}
  \int_{S_\phi(x, 2\theta_{\ast}t)} g_{S_\phi(x, 2\theta_{\ast}t)}^p (y, x) dy\leq 
  C_0^p|S_\phi(x, 2\theta_{\ast}t)| + \frac{pC_3}{\frac{n}{n-2}-p} C_0 ^{p- \frac{n}{n-2}} \leq C(n, \lambda,\Lambda, \rho, p) t^{\frac{n}{2}-\frac{n-2}{2}p}
  \label{gsec}
\end{equation}
by the upper bound on the volume of sections in Lemma \ref{global-gradient}. Since $t\leq c_1$, we can use the  lower 
bound on the volume of sections in Lemma \ref{global-gradient} to deduce from (\ref{upg}), (\ref{gsec}) and $p<\frac{n}{n-2}$ that
$$\sup_{x\in S_\phi(x_0, t)\cap\Omega} \int_{S_\phi(x_0, t)} g_{S_\phi(x_0, t)}^p (y, x) dy \leq C(n, \lambda,\Lambda, \rho, p) t^{\frac{n}{2}-\frac{n-2}{2}p}
\leq C(n, \lambda,\Lambda, \rho, p) |S_\phi(x_0, t)|^{1-\frac{n-2}{n}p}.$$
\end{proof}

\subsection{Proofs of the properties of the Green's function}
\begin{proof} [Proof of Lemma \ref{abplem}] Let $x_0\in V$ be given and $\sigma= g_{V}(\cdot, x_0)$.
 By the ABP estimate, for any $\varphi\in L^n(V)$, the solution $\psi$ to 
$$-\Phi^{ij}\psi_{ij}=\varphi~\text{in}~V,~\psi=0~\text{on}~\p V,$$
satisfies
$$|\int_V \sigma(x) \varphi(x) dx|=|\psi(x_0)|\leq C(n)|V|^{1/n}\left\|\frac{\varphi}{\det \Phi}\right\|_{L^n(V)} \leq C(n,\lambda) |V|^{1/n} \|\varphi\|_{L^n(V)}.$$
Here we used the identity $\det \Phi= (\det D^2 \phi)^{n-1}$ and $\det D^2\phi\geq \lambda$. By duality, we obtain
$$\left(\int_V \left[g_{V}(x, x_0)\right]^{\frac{n}{n-1}} dx\right)^{\frac{n-1}{n}}=\|\sigma\|_{L^{\frac{n}{n-1}}(V)}\leq C(n,\lambda) |V|^{1/n}.$$
The uniform $L^1$ bound for $g_V(\cdot, x_0)$ on $V$ follows from the H\"older inequality and the uniform $L^{\frac{n}{n-1}}$ bound.
\end{proof}
\begin{proof}[Proof of Lemma \ref{genup}] 
Let $c_1$ and $C_1$ be as in Lemma \ref{vol_lem}. Denote by $$D= (S_\phi(x_0, t)\backslash S_\phi (x_0, r_2 t)) \cup S_\phi (x_0, r_1 t)$$ where $0<r_1<1/2<r_2<1$. Then, 
by \cite[Lemma 6.5.1]{G} and Lemma \ref{vol_lem}, we can estimate
$$|D|\leq n(1-r_2)|S_\phi (x_0, t)| + |S_\phi(x_0, r_1 t)|\leq C_1 n(1-r_2)t^{n/2} + C_1(r_1 t)^{n/2}\leq (c_1/2)^n t^{n/2}$$ 
if
$r_1, 1-r_2$ are universally small. Then by Lemma \ref{vol_lem},
\begin{equation}
\frac{c_1}{2} t^{n/2}\leq |S_\phi(x_0, t)\backslash D|.
\label{notK}
\end{equation}
Given $0<r_1<r_2<1$ as above, 
we have
\begin{equation}\sup_{S_\phi(x_0, t)\backslash D} \sigma \leq C(n,\lambda,\Lambda) \inf_{S_\phi(x_0, t)\backslash D} \sigma.
\label{CGK}
\end{equation}
Combining (\ref{notK}) and (\ref{CGK}) with the $L^1$ bound on $\sigma$, we find that
$$\sigma(x) \leq C(n,\lambda,\Lambda)At^{-\frac{n}{2}}~\forall x\in S_\phi(x_0, t)\backslash D.$$
 Since $r_2>1/2>r_1$, we obtain the desired upper bound for $\sigma(x)= g_V(x, x_0)$ when $x\in\p S_\phi(x_0, t/2)$.

For completeness, we include the details of (\ref{CGK}). By \cite[Theorem 3.3.10]{G}, we can find a universal $\alpha\in (0, 1)$ such that for each 
$x\in S_\phi(x_0, t)\backslash D$, the section $S_\phi(x, \alpha t)$ satisfies
$$x_0\not\in S_\phi(x,\alpha t)~\text{and}~ S_\phi(x, \alpha t)\subset S_\phi (x_0, t).$$ Using Lemma \ref{cov_lem}, we can find a collection 
of sections $S_\phi(x_i, \tau\alpha t)$ with $x_i\in S_\phi(x_0, t)\backslash D$ such that
$$S_\phi(x_0, t)\backslash D \subset \bigcup_{i\in I} S_\phi(x_i,\tau\alpha t)$$
and
$S_\phi(x_i,\delta\tau\alpha t)$ are disjoint for some universal $\delta\in (0, 1).$ By using the volume estimates in Lemma \ref{vol_lem}, we find that $|I|$ is universally 
bounded. Now, we apply the Harnack inequality in Theorem \ref{Holder_thm} to $\sigma$ in  each $S_\phi (x_i, \alpha t)$ to obtain (\ref{CGK}).
\end{proof}

\begin{proof}[Proof of Lemma \ref{doubleh}] Let $\sigma =  g_{S_\phi(x_0, 2t)}(\cdot, x_0).$
 We consider two cases for a universally large $K$.\\
 {\bf Case 1:} $S_\phi(x_0, t/K)\subset\Omega$. 
 By H\"older inequality and Lemma \ref{abplem}, we have
 \begin{eqnarray*}
  \int_{S_\phi(x_0, t/K)} \sigma \leq \|\sigma\|_{L^{\frac{n}{n-1}}(S_\phi(x_0, 2t))} |S_\phi(x_0, t/K)|^{\frac{1}{n}}\leq C(n,\lambda)
  |S_\phi(x_0, 2t)|^{\frac{1}{n}} |S_\phi(x_0, t/K)|^{\frac{1}{n}} \leq C t.
 \end{eqnarray*}
We used the volume estimate in Lemma \ref{vol_lem} in the last inequality. 
By Lemma \ref{genup}, we have
$$\max_{x\in\p S_\phi(x_0, \frac{t}{2K})} \sigma(x) \leq C t^{-\frac{n-2}{2}}.$$
Now, the
bound (\ref{dh}) follows from the maximum principle.\\
 {\bf Case 2:} $S_\phi(x_0, t/K)\not\subset\Omega$. Let $\bar\theta=2\theta_{\ast}^2, K=\bar\theta^6$.
 Then $c\geq t\geq K\bar h(x_0).$ Suppose that $\p S_\phi(x_0, \bar h(x_0))\cap\p\Omega=0.$ 
By the engulfing property in Theorem \ref{engulfing2} and $\bar\theta>\theta_{\ast}$, we have
\begin{equation*}S(x_0, \frac{t}{K})\subset S(0, \frac{\bar \theta t}{K})\subset S(x_0, \frac{\bar \theta^2t}{K})\subset S(0, \frac{\bar \theta^3t}{K})\subset S(x_0, \frac{\bar\theta^4t}{K})
\subset S(0, \frac{\bar\theta^5t}{K}) \subset S(x_0, \frac{\bar\theta^6t}{K})\subset S(x_0, t).
\end{equation*}
By Lemma \ref{cen-lem}, there is $y\in \p S_\phi(0, \frac{\bar \theta^3 t}{K})\cap\Omega$ such that $dist(y,\p\Omega)\geq ct^{1/2}$. Hence,  from Proposition 
\ref{tan_sec}, we can find a universal constant $c_2\leq \frac{\bar \theta}{\theta_{\ast}K}$ 
such that $S_{\phi}(y, c_2t)\subset\Omega.$ Then
\begin{equation}S(0, \frac{\bar \theta t}{K}) \cap S_\phi(y, c_2t)=\emptyset; S_\phi(y, c_2t)\subset S(0, \frac{\bar \theta^5 t}{K})\cap\Omega.
 \label{goody-sec}
\end{equation}
Indeed, 
suppose there is $z\in S(0, \frac{\bar \theta t}{K})\cap S_\phi(y, c_2t)$. Then, recalling the quasi-distance $\delta_\phi$ in (\ref{del_dist}), we have
$$\delta_\phi(0, y)=\frac{\bar\theta^3 t}{K}, \delta_\phi(0, z)\leq \frac{\bar \theta t}{K},~
\text{and}~\delta_\phi(y, z)\leq c_2t.$$
On the other hand, we observe from the engulfing property in Theorem \ref{engulfing2} that 
$$\delta_\phi(z, y)\leq \theta_{\ast}\delta_\phi(y, z)\leq \theta_{\ast}c_2t.$$
By Lemma \ref{quasi}, we find that
$$\delta_\phi(0, y)\leq \theta_\ast^2 (\delta_\phi(0, z)+\delta_\phi(z, y))\leq \theta_\ast^2 [\frac{\bar \theta t}{K} + \theta_\ast c_2t]\leq \frac{\bar \theta}{2}[\frac{\bar \theta t}{K}
+ \frac{\bar \theta t}{K}]=\frac{\bar \theta^2 t}{K}<\frac{\bar \theta^3 t}{K}.$$
This contradicts $\delta_\phi(0, y)= \frac{\bar \theta^3 t}{K}.$

Suppose now that there is $z\in \p S(0, \frac{\bar \theta^5 t}{K})\cap S_\phi(y, c_2t)$. Then, from $S_\phi(y, c_2 t)\subset\Omega$, we have
$\delta_\phi(0, z)= \frac{\bar \theta^5 t}{K}$ and $\delta_\phi(y, z)\leq c_2 t \leq \frac{\bar\theta t}{K}$. 
But from Lemma \ref{quasi}, we have a contradiction because
$$\delta_\phi(0, z)\leq \theta_\ast^2 (\delta_\phi(0, y)+\delta_\phi(y, z))\leq \theta_\ast^2 [\frac{\bar \theta^3 t}{K} + 
\frac{\bar \theta t}{K}]\leq \frac{\bar \theta}{2}[\frac{\bar \theta^3 t}{K}
+ \frac{\bar \theta^3 t}{K}]<\frac{\bar \theta^5 t}{K}=\delta_\phi(0, z).$$
From (\ref{goody-sec}), the $L^{\frac{n}{n-1}}$ bound on the Green's function in Lemma \ref{abplem}, and Lemma \ref{vol_lem}, we obtain
 \begin{eqnarray*}
 \int_{S_\phi(y, c_2t)} \sigma(x) dx\leq 
  \int_{S_\phi(x_0, 2t)} \sigma (x) dx &\leq& \|\sigma\|_{L^{\frac{n}{n-1}}(S_\phi(x_0, 2t))} |S_\phi(x_0, 2t)|^{\frac{1}{n}}\\ &\leq& C(n,\lambda)
  |S_\phi(x_0, 2t)|^{\frac{1}{n}} |S_\phi(x_0, 2t)|^{\frac{1}{n}} \leq C t.
 \end{eqnarray*}
 By Lemma \ref{genup}, we have
$\max_{x\in\p S_\phi(y, c_2t/2)} \sigma(x) \leq C t^{-\frac{n-2}{2}}.$
Thus, by the maximum principle, we have at the {\it apogee} $y$ of $S(0, \frac{2\bar\theta^3t}{K})$ the following estimate
$$\sigma(y)\leq Ct^{-\frac{n-2}{2}}.$$
With this estimate, we use Lemma \ref{chain_lem} and the boundary Harnack inequality in Theorem \ref{carl} 
together with the interior Harnack inequality in Theorem \ref{Holder_thm} to conclude that $\sigma$ is bounded by $Ct^{-\frac{n-2}{2}}$ on 
$ \p S(0, \frac{\bar\theta^3t}{K})\cap\Omega$. Hence, the
bound (\ref{dh}) follows from the maximum principle.
\end{proof}
\begin{proof}[Proof of Lemma \ref{todh}]
To prove (\ref{iter_max}), we consider
$$w(x) = g_\Omega(x, x_0)- g_{S_\phi(x_0, 2t)}(x, x_0).$$
It satisfies
$L_\phi w=0~\text{in}~ S_\phi(x_0, 2t)\cap \Omega.$
In $\overline{S_\phi(x_0, 2t)}$, $w$ attains its maximum value on the boundary $\p S_\phi(x_0, 2t)$. Thus, for $x\in \p S_\phi(x_0, t)$, we have
$$g_\Omega(x, x_0)- g_{S_\phi(x_0, 2t)}(x, x_0)\leq \max_{z\in \p S_\phi(x_0, 2t)} w(z)= \max_{z\in \p S_\phi(x_0, 2t)} g_\Omega(z, x_0)
$$
since $g_{S_\phi(x_0, 2t)}(z, x_0)=0$ for $z$ on $\p S_\phi(x_0, 2t)$. This together with the Lemma \ref{doubleh} gives
\begin{eqnarray*}\max_{x\in \p S_\phi(x_0, t)} g_\Omega(x, x_0)&\leq& \max_{z\in \p S_\phi(x_0, t)} g_{S_\phi(x_0, 2t)}(z, x_0) + 
\max_{z\in \p S_\phi(x_0, 2t)} g_\Omega(z, x_0)\\ &\leq & Ct^{-\frac{n-2}{2}} + \max_{z\in \p S_\phi(x_0, 2t)} g_\Omega(z, x_0).
\end{eqnarray*}
Therefore, (\ref{iter_max}) is proved.
\end{proof}

{} 
\end{document}